\documentclass[10pt,a4paper,reqno]{amsart}
\usepackage{amsthm,amssymb,amstext,graphicx}
\usepackage{bbm}
\usepackage{enumerate}
\usepackage{amscd}
\usepackage{amssymb}
\usepackage{amsthm}
\usepackage{mathtools}
\usepackage{hyperref}
\usepackage{empheq}
\usepackage[utf8]{inputenc}
\usepackage{lmodern}
\usepackage{enumitem}
\usepackage[capitalise]{cleveref}
\usepackage{hyperref}
\hypersetup{
  colorlinks=true,
  pdfpagelabels=true,
  pdfpagelayout==TwoColumnRight,
  pdftitle={},
  pdfauthor={© Sven Karbach},
  pdfsubject={Project2},
  pdfkeywords={}
}

\DeclareMathOperator\dom{dom}
\DeclareMathOperator\Tr{Tr}

\DeclareMathOperator\lin{lin}
\newcommand*\D{\mathop{}\!\mathrm{d}}
\newcommand*\E{\mathop{}\!\mathrm{e}}
\newcommand*\I{\mathop{}\!\mathrm{i}}

\newtheorem{theorem}{Theorem}[section]
\newtheorem*{theorem*}{Theorem}
\newtheorem{lemma}[theorem]{Lemma}
\newtheorem{proposition}[theorem]{Proposition}
\newtheorem{corollary}[theorem]{Corollary}
\newtheorem*{corollary*}{Corollary}
\theoremstyle{plain}

\theoremstyle{definition}
\newtheorem{definition}[theorem]{Definition}
\newtheorem*{definition*}{Definition}

\newtheoremstyle{remark}
  {.2\baselineskip}
  {.2\baselineskip}
  {\normalfont}
  {}
  {\bfseries}
  {\ifx\thmnote\@gobble.\else\normalfont.\fi}
  {.5em}
  {}
\theoremstyle{remark}
\newtheorem{remark}[theorem]{Remark}

\newtheoremstyle{example}
  {.3\baselineskip}
  {.3\baselineskip}
  {\normalsize}  
  {0pt}       
  {\bfseries} 
  {.}         
  {5pt plus 1pt minus 1pt} 
  {}          
\theoremstyle{example}

\newtheorem*{assumption*}{\assumptionnumber}
\providecommand{\assumptionnumber}{}
\makeatletter
\newenvironment{assumption}[2]
 {%
  \renewcommand{\assumptionnumber}{Assumption #1$\mathfrak{#2}$}%
  \begin{assumption*}%
  \protected@edef\@currentlabel{#1$\mathfrak{#2}$}%
 }
 {%
  \end{assumption*}
 }

\newlist{theoremenum}{enumerate}{1}
\setlist[theoremenum]{label=\roman*), ref=\textup{\thetheorem~\roman*)}}
\crefalias{theoremenumi}{theorem}

\newlist{propenum}{enumerate}{1}
\setlist[propenum]{label=\roman*), ref=\textup{\theproposition~\roman*)}}
\crefalias{propenumi}{proposition}

\newlist{defenum}{enumerate}{1}
\setlist[defenum]{label=\roman*), ref=\textup{\thedefinition~\roman*)}}
\crefalias{defenumi}{definition}

\newlist{lemenum}{enumerate}{1}
\setlist[lemenum]{label=\roman*), ref=\textup{\thelemma~\roman*)}}
\crefalias{lemenumi}{lemma}

\newlist{remenum}{enumerate}{1}
\setlist[remenum]{label=\roman*), ref=\textup{\theremark~\roman*)}}
\crefalias{remenumi}{remark}

\renewcommand{\MR}{\mathbb{R}}
\newcommand{\MC}{\mathbb{C}}

\newcommand{\MN}{\mathbb{N}}
\newcommand{\MP}{\mathbb{P}}

\newcommand{\MS}{\mathbb{S}}
\newcommand{\MF}{\mathbb{F}}
\newcommand{\cF}{\mathcal{F}}

\newcommand{\cB}{\mathcal{B}}

\newcommand{\cD}{\mathcal{D}}
\newcommand{\cE}{\mathcal{E}}
\newcommand{\cH}{\mathcal{H}}

\newcommand{\cL}{\mathcal{L}}

\newcommand{\cV}{\mathcal{V}}

\newcommand{\cG}{\mathcal{G}}

\newcommand{\sP}{\mathsf{P}}

\newcommand{\df}{\coloneqq}
\newcommand{\one}{\mathbbm{1}}
\newcommand{\interior}[1]{{\kern0pt#1}^{\mathrm{o}}}
\newcommand{\set}[1]{\left\{ #1\right\}}
\newcommand{\norm}[1]{\|#1\|}

\newcommand{\EXspec}[2]{\mathbb{E}_{#1}\left[#2\right]}

\newcounter{Task}

\newcommand{\cHplus}{\cH^{+}}

\newcommand{\cHpluso}{\cHplus\setminus \{0\}}
\newcommand{\MRplus}{\MR^{+}}
\newcommand{\dm}{m(\D\xi)}
\newcommand{\dmu}{\mu(\D\xi)}

\newcommand{\sgcsven}{\begin{color}{blue}}
\newcommand{\cgssven}{\end{color}}

\newcommand{\be}{\mathbf{e}}
\newcommand{\bP}{\mathbf{P}}

\begin{document}

\title[]{Finite-rank approximation of affine processes on positive
  Hilbert-Schmidt operators}    
\author{Sven Karbach}
\address{Department of Mathematics, University of Hamburg, Germany} 
\email{sven@karbach.org}
\thanks{The research for publication of this work received financial
  assistance from \emph{The Dutch Research Council} (NWO)
  (Grant No: C.2327.0099)}

\begin{abstract}
  In this article, we present a method for approximating affine processes on the
  cone of positive Hilbert-Schmidt operators using matrix-valued affine
  processes. By leveraging results from the theory on affine processes
  with values in the cone of symmetric and positive semi-definite matrices, we construct
  sequences of finite-rank operator-valued affine processes that converge weakly
  to the target processes and provide convergence rates for their Laplace
  transforms using Galerkin approximations of the associated operator-valued
  generalized Riccati equations. This article not only offers a practical
  approximation scheme for operator-valued affine processes with error bounds
  that hold uniformly in time, but also provides a novel existence proof for this
  class of affine processes with c\`adl\`ag paths, including affine pure-jump
  processes with infinite variation and state-dependent jump
  intensities. In addition to the theoretical significance, the results of this paper
  provide useful tools for analyzing and approximating infinite-dimensional
  affine stochastic covariance models that were recently introduced in
  mathematical finance.   
\end{abstract}

\maketitle{}

\section{Introduction}\label{sec:kar22-introduction}
\label{eq:kar22-exponential-affine-form}
In this article, we introduce and study finite-dimensional approximation of
affine processes with values in the cone of positive self-adjoint
Hilbert-Schmidt operators. In line with the conventions in the literature~\cite{CFMT11} we call a stochastically continuous Markov process \emph{affine}, whenever its cumulant generating function, at any time-point, is
affine in the initial value of the process and can be explicitly determined up
to the solution of a pair of associated \emph{generalized Riccati equations}.\par{}

The appeal of the affine class lies in its good tractability, as the cumulant
generating functions have a quasi-explicit form, making it a popular choice for
models in finance see, e.g.,~\cite{DFS03, Kel08, cuchiero2011affine} and the
references therein. In the last two decades affine processes and their
applications have been extensively studied by many authors on various state
spaces see, e.g.,~\cite{CFMT11, DFS03, CMET16, DL06}, including the
\emph{canonical state space} $\MR_+^{n}\times\MR^{d}$ in~\cite{DFS03} and the
cone of positive and symmetric $d\times d$-matrices $\MS_{d}^{+}$
in~\cite{CFMT11}.\par{}
In recent years there has been a growing interest in infinite-dimensional
versions of affine processes and their applications as seen in works such
as~\cite{Gra16, STY20, CT20, CKK22a}, including affine diffusion on canonical
state spaces in Hilbert spaces in~\cite{STY20}, applied to
term-structure modeling in~\cite{Yu17}, affine Markovian lifts of stochastic Volterra
equations in~\cite{CT20}, relevant for rough volatility modeling, see,
e.g,~\cite{JEE19}, and affine pure-jump processes on positive Hilbert-Schmidt
operators in~\cite{CKK22a}, that can be used as instantaneous covariance
processes in infinite-dimensional stochastic covariance models, see~\cite{BRS18,
  BS18, BLDP22, CKK22b, FK22, Kar22}. The class of affine processes on positive
Hilbert-Schmidt operators can be viewed as the natural infinite-dimensional
extension of the well-studied class of affine processes on $\MS_{d}^{+}$ and
both classes coincide for positive Hilbert-Schmidt operators defined on
the Euclidean space $\MR^{d}$.\par{}

In the present article, we go further: We demonstrate that all affine processes
on positive Hilbert-Schmidt operators essentially emerge as weak limits of
sequences of $\MS_{d}^{+}$-valued affine processes as the dimension $d$
increases towards infinity. By proving our main result
(Theorem~\ref{thm:kar22-main-convergence} below) we not only provide a tractable
approximation method for affine processes on positive Hilbert-Schmidt operators,
that enhance our understanding of the relationship between finite and
infinite-dimensional operator-valued affine processes, but we also uncover some
fundamental structural differences that arise in the transition from a finite to
an infinite dimensional setting.  

\subsection{Contributions and related literature}

The main contribution of this article is the introduction of a novel method for
approximating affine processes on positive Hilbert-Schmidt operators using
finite-rank operator-valued affine processes and Galerkin-type approximations of
the associated \emph{operator-valued generalized Riccati equations}. Our method is both tractable and constructive, providing a proof for
the existence of a broad class of affine processes on positive Hilbert-Schmidt
operators with c\`adl\`ag paths. While the existence of these processes was
previously established in~\cite{CKK22a}, the path regularity of the affine class
was left as an open problem. In the present article, we solve this by
establishing the approximating in the Skorohod space of all c\`adl\`ag path. In the following paragraphs, we provide a more detailed outline
of our contributions and related literature:

\vspace{-1mm}

\subsubsection*{Galerkin approximation of generalized Riccati equations}
In Proposition~\ref{prop:kar22-existence-Riccati} below, we construct Galerkin
type approximations of the solutions to the operator-valued generalized Riccati
equations~\eqref{eq:kar22-Riccati}, which are modulated by a
so-called \emph{admissible parameter set} (see Definition~\ref{def:kar22-admissible}
below), that uniquely identifies an affine process. The Galerkin approximations are defined on positive
finite-rank operators and we prove their convergence to the solutions of the
original generalized Riccati equations uniformly on compact time
intervals. Moreover, denoting the sequence of Galerkin approximations by
$(\phi_{d})_{d\in\MN}$ and $(\psi_{d})_{d\in\MN}$ and the original solutions by
$\phi$ and $\psi$ we present explicit bounds, in terms of the admissible
parameters and the initial value $u$ of $\psi$, for the following
approximation error:   
\begin{align*} \sup_{t\in[0,T]}\big(|\phi_{d}(t,\bP_{d}(u))-\phi(t,u)|+\norm{\psi_{d}(t,\bP_{d}(u))-\psi(t,u)}\big),  
\end{align*}
where $\norm{\cdot}$ denotes the Hilbert-Schmidt norm and $(\bP_{d})_{d\in\MN}$
are specific projections onto subspaces of self-adjoint operators of rank
$d\in\MN$ defined on some underlying Hilbert space. Galerkin approximation of Riccati equations on Hilbert-Schmidt operators have
been previously studied in the literature, as they are a fundamental tool in
stochastic control and filtering theory see, e.g.,~\cite{Ros91}. Our work
extends this literature by considering Galerkin approximation of
\emph{generalized} Riccati equations, that admit for non-linear components given
by integrals of vector-valued measures, and by quantifying the approximation
error of such equations through error bounds. 

\vspace{-1mm}

\subsubsection*{Finite-rank operator-valued affine processes}

For every finite rank $d$, we show the existence of an affine process on the cone of
positive operators with rank at most $d$, which can be associated with the Galerkin
approximations $\phi_{d}$ and $\psi_{d}$ from before. This intermediate step is
presented by Proposition~\ref{prop:kar22-embedding-affine-main} below and
yields, as a convenient byproduct, the existence of affine processes on positive
finite-rank operators, which are similar, but not equivalent, to their
matrix-valued counterparts in~\cite{CFMT11}. 

\vspace{-1mm}

\subsubsection*{Existence and weak convergence} Our main
result, Theorem~\ref{thm:kar22-main-convergence}, shows that the sequence
of \emph{finite-rank operator-valued affine processes} described above, denoted
by $(X^{d})_{d\in\MN}$, is \emph{tight} on $D(\MRplus,\cHplus)$, the Skorohod space
of all c\`adl\`ag paths from $\MRplus$ into the cone of positive Hilbert-Schmidt
operators $\cHplus$. Moreover, we prove that
the processes $(X^{d})_{d\in\MN}$ solve the \emph{martingale problem} for an
associated sequence of Kolmogorov-type operators. From this and tightness of the
sequence, we derive the weak convergence of $(X^{d})_{d\in\MN}$ to a unique
affine c\`adl\`ag Markov process $(X,\MP_{x})$ with values in $\cHplus$ and
present explicit convergence rates for the associated Laplace transforms. In
addition, we provide a convenient semimartingale representation for this affine
class, give concrete examples of operator-valued affine processes and their
approximations and show that in contrast to the matrix-valued case~\cite{May12},
infinite-rank operator-valued processes admit jumps of infinite
variation. Finite-dimensional approximations of affine diffusion in a Hilbert
space setting were already discussed in~\cite{Yu17}. However, the approximation
method was not used for proving the existence of affine diffusion and also no
explicit convergence rates for the Laplace transforms of the processes or their
associated generalized Riccati equations were established.

\vspace{-1mm}

\subsubsection*{Applications in affine stochastic covariance modeling}
The presented approximation method paves the way for improving the computational
efficiency and extending the range of applications of affine processes with values
in the cone of positive Hilbert-Schmidt operators with possibly infinite
rank. This is particularly relevant for infinite-dimensional affine stochastic
volatility models that were recently introduced in~\cite{CKK22b}. In a companion
article, we will examine finite-rank approximations of affine
stochastic covariance models using tools established in the present work.

\vspace{-1mm}

\subsection{Layout of the article}
In Section~\ref{sec:kar22-notat-prel}
we introduce our notation and recall some preliminaries on affine
processes. Section~\ref{sec:kar22-main-results} is devoted to the presentation of our main results. More specifically, in Section~\ref{sec:galerk-appr-gener} we
introduce Galerkin approximations of the generalized Riccati equations and
provide explicit convergence rates, in Section~\ref{sec:finite-rank-operator} we
state our results on the existence of finite-rank operator-valued affine
processes associated with the Galerkin approximations and in Section~\ref{sec:exist-weak-conv} we present a comprehensive version of our
main result on the existence and approximation of affine processes on
positive Hilbert-Schmidt operators. To illustrate our main findings we give in
Section~\ref{sec:kar22-regul-affine-proc} a concrete example of
an affine process on positive Hilbert-Schmidt operators of infinite-variation
and its finite-rank approximations. The proofs are contained in the
subsequent four chapters: In Section~\ref{sec:proof-prop-refpr} we prove existence and convergence of the Galerkin approximations, in Section~\ref{sec:kar22-affine-finite-rank} we construct associated sequences of
finite-rank operator-valued affine processes and in
Section~\ref{sec:kar22-tightness-weak-convergence} we show weak
convergence of the sequences of finite-rank processes.


\section{Notation and preliminaries}\label{sec:kar22-notat-prel}

\subsection{Notation}\label{sec:kar22-notation}
We set $\MN_{0}=\set{0,1,2,\ldots}$ and $\MN=\set{1,2,\ldots}$. For a complex
number $z=a+\I b\in\MC$ we denote its real part $a$ by $\Re(z)$ and its
imaginary part $b$ by $\Im(z)$. For a vector space $X$ and a subset
$U\subseteq X$ we denote the linear span of $U$ in $X$ by $\lin(U)$. For a
Banach space $X$ with norm $\norm{\cdot}_{X}$, we denote by $\cL(X)$ the space
of all \emph{bounded linear operators} on $X$, which becomes a Banach space
when equipped with the operator norm $\| T \|_{\cL(X)}=\sup_{\norm{x}_{X}\leq
  1}\norm{ T x}_{X}$ for $T\in\cL(X)$. Throughout this article we let $(H,
\langle \cdot,\cdot\rangle_H)$ be a real separable Hilbert space and we denote its norm by
$\norm{\cdot}_{H}$. Moreover, let $(V,(\cdot,\cdot)_{V})$ be a second separable
Hilbert space with norm $\norm{\cdot}_{V}$, then we denote the space of all
\emph{Hilbert-Schmidt operators} mapping from $H$ to $V$ by
$\cL_{2}(H,V)$. The space $\cL_{2}(H,V)$ is a Hilbert space itself when
equipped with the inner product $\langle\cdot,\cdot\rangle_{\cL_{2}(H,V)}$,
which for $A,B\in\cL_{2}(H,V)$ is defined by   
  $\langle A, B \rangle_{\cL_2(H,V)} \df \Tr(B^{*}A)=\sum_{n=1}^{\infty} \langle A e_n, B e_n \rangle_V,$
where $(e_n)_{n\in \MN}$ is an orthonormal basis of $H$ and the definition is
independent of the choice of the basis, see, e.g.~\cite[Section VI.6]{Wer00}.
Whenever $H=V$ we simply write $\cL_{2}(H)\df \cL_{2}(H,H)$, denote the inner
product by $\langle \cdot, \cdot \rangle$ and the norm by $\| \cdot
\|$. The subspace of all self-adjoint Hilbert-Schmidt operators on $H$ is
denoted by $\cH$ and we let $\cH^+$ stand for the set of all \emph{positive
  operators} in $\cH$, i.e.
$$\cH^{+} \df \{ A \in \cH \colon \langle Ah, h\rangle_H \geq 0 \text{ for all
} h\in H \}.$$ Note that $\cHplus$ is a closed \emph{convex cone} in $\cH$,
i.e. it is closed, $\cHplus+\cHplus\subseteq \cHplus$, $\lambda\cHplus \subseteq
\cHplus$ for all $\lambda\geq 0$ and $\cHplus\cap(-\cHplus) = \{ 0\}$. The cone $\cHplus$ induces a partial
ordering ``$\leq_{\cHplus}$'' on $\cH$ and we write $x\leq_{\cHplus} y$ whenever
$y-x\in \cHplus$. The cone $\cHplus$ is \emph{generating} for $\cH$,
i.e. $\cH=\cHplus -\cHplus$ and \emph{monotone},
i.e. $0\leq_{\cHplus}x\leq_{\cHplus} y$ implies $\norm{x}\leq\norm{y}$,
see~\cite{CKK22a}. We define $D(\MRplus,\cHplus)$ to be the space of all
c\`adl\`ag path from $\MRplus$ into $\cHplus$ equipped with the Skorohod
topology, see~\cite{Jak86}. For
any $A\in\cL(V)$ we denote by $A^{*}$ the \emph{adjoint} of $A$. For two elements
$x$ and $y$ in $V$ we define the operator $x\otimes y\in \cL(V)$ by $(x\otimes
y)h=\langle x,h \rangle_{V} y$ for every $h\in V$ and write $x^{\otimes 2}\df
x\otimes x$. If $V\subset H$ we say that $V$ is \emph{continuously embedded}
in $H$, if there exists a constant $C$ such that $\norm{u}_{H}\leq
C\norm{u}_{V}$ for all $u\in V$. If in addition the embedding operator of $V$
into $H$ is compact, then we say that $V$ is \emph{compactly embedded} into
$H$ and write $V\subset\!\subset H$.  

\subsection{Finite-rank projection schemes for Hilbert-Schmidt operators}\label{sec:kar22-proj-schem-hilb}
Let $(e_{i})_{i\in\MN}$ be an orthonormal basis of $H$ which can be chosen
arbitrarily, but is fixed throughout the section. For every
$d\in\MN$ we denote by $H_{d}$ the $d$-dimensional subspace of $H$ spanned by
the first $d$ basis vectors, i.e. $$H_{d}\df \lin\set{e_{i}\colon i=1,\ldots,d}.$$
We denote the orthogonal projection of $H$ onto $H_{d}$, with respect to the inner
product $\langle \cdot,\cdot\rangle_{H}$, by $\sP_{d}$. For every $i\in\MN$ we set
$\be_{i,i}\df e_{i}\otimes e_{i}$ and for all $i\neq j$ set $\be_{i,j}\df
\frac{1}{\sqrt{2}}(e_{i}\otimes e_{j}+e_{j}\otimes e_{i})$. Note that
$\norm{\be_{i,j}}=1$, $\be_{i,j}=\be_{j,i}$ for every $i,j\in\MN$ and it
can be seen that the family $\set{\be_{i,j}}_{i\leq j\in\MN}\df\set{\be_{i,j}\colon i,j\in\MN, i\leq j}$ is an orthonormal basis of $\cH$. For every $d\in\MN$, we let
$\cH_{d}$ stand for the finite-dimensional subspace of $\cH$ spanned by the
family $\set{\be_{i,j}\colon\, 1\leq i\leq j\leq d\,}$, i.e. $$\cH_{d}\df \lin\set{\be_{i,j}\colon\, 1\leq i\leq j\leq d\,}.$$
We denote the orthogonal projection of $\cH$ onto $\cH_{d}$, with respect to the inner
product $\langle\cdot,\cdot\rangle$, by $\bP_{d}$ and note that for every
$d\in\MN$ and $u\in\cH$ we have $\bP_{d}(u)=\sP_{d}u\sP_{d}$. Moreover, every
operator in $\cH_{d}$ is self-adjoint and of rank at most $d$. We write
$\bP_{d}^{\perp}(u)\df u-\bP_{d}(u)$ and note that
$\lim_{d\to\infty}\norm{\bP^{\perp}_{d}x}=0$ for all $x\in\cH$. In addition, it can be seen that
$\cH_{d}=\set{u\sP_{d}\colon \,u\in \cL(H_{d}),\, u=u^{*}}$ and for the
cone of all positive self-adjoint operators in $\cH_{d}$, denoted by $\cH^{+}_{d}$, we have
\begin{align*}
  \cH^{+}_{d}\df\set{u\sP_{d}\colon\; u\in\cL(H_{d}),\, u=u^{*},\, (u h,h)_{H}\geq 0\,\,\forall h\in H_{d}}.
\end{align*}
Note further that $\cHplus_{d}\subseteq\cHplus_{d+1}\subseteq\cHplus$ for all
$d\in\MN$. For more details on the subspace of finite-rank operators in the
ambient space of all Hilbert-Schmidt operators see~\cite{Ros91}. As
in~\cite{Goe84}, we call a sequence $(\cH_{d},\bP_{d})_{d\in\MN}$ defined as
above a \emph{projection scheme} in $\cH$ (with respect to the orthonormal
basis $\set{\be_{i,j}}_{i\leq j\in\MN}$).  

\subsection{Affine processes, admissible parameters and the generalized Riccati equations}\label{sec:kar22-admiss-param-gener}
As before, let $\left(H,\langle \cdot,\cdot\rangle_{H}\right)$ be a real
separable Hilbert space and $\cHplus$ the cone of positive self-adjoint  Hilbert-Schmidt
operators on $H$. Consider an $\cHplus$-valued time-homogeneous Markov process
$\left(X,(\MP_{x})_{x\in\cHplus}\right)$, where $\MP_{x}$ represents the
distribution of $X$ given that $X_{0}=x$. This process is called \textit{affine},
if its Laplace transform is of an exponential affine form in the initial value
$X_{0}=x\in\cHplus$, i.e. if 
\begin{align}\label{eq:kar22-affine}
  \EXspec{\MP_{x}}{\E^{-\langle \xi,u\rangle}}=\E^{-\phi(t,u)-\langle
  x,\psi(t,u)\rangle},\quad t\geq 0,\, u\in\cHplus\,,
\end{align}
for some functions $\phi\colon \MRplus\times\cHplus\to\MRplus$ and $\psi\colon
\MRplus\times \cHplus \to \cHplus$. Affine processes on $\cHplus$ were first
introduced and studied in~\cite{CKK22a} and can be uniquely identified by a parameter
tuple $(b,B,m,\mu)$, which for the readers convenience we recall
from~\cite[Definition 2.3]{CKK22a}: First, define the \emph{truncation function}
$\chi\colon\cH\to\cH$ by $\chi(\xi)\df\xi\one_{\set{\norm{\xi}\leq 1}}(\xi)$
for $\xi\in\cH$. Next, recall the following definition.
\begin{definition}\label{def:kar22-admissible}
  An \emph{admissible parameter set} $(b,B,m,\mu)$ consists of
  \begin{defenum}
  \item\label{item:kar22-m-2moment} a measure
    $m\colon\cB(\cHpluso)\to [0,\infty]$ such that
    \begin{enumerate}
    \item[(a)] $\int_{\cHpluso} \| \xi \|^2 \,\dm < \infty$ and
    \item[(b)] $\int_{\cHpluso}|\langle\chi(\xi),h\rangle|\,\dm<\infty$ for all $h\in\cH$  
   and there exists an element $I_{m}\in \cH$ such that $\langle
   I_{m},h\rangle=\int_{\cHpluso}\langle \chi(\xi),h\rangle\, m(\D\xi)$ for
   every $h\in\cH$;  
    \end{enumerate}   
  \item \label{item:kar22-drift} a vector $b\in\cH$ such that
    \begin{align*}
      \langle b, v\rangle - \int_{\cHpluso} \langle
      \chi(\xi), v\rangle \,m(\D\xi) \geq 0\, \quad\text{for all}\;v\in\cHplus;
    \end{align*}    
  \item \label{item:kar22-affine-kernel} a $\cH^{+}$-valued measure $\mu \colon
    \mathcal{B}(\cHpluso) \rightarrow \cH^+$ such that the kernel 
    $M(x,\D\xi)$, for every $x\in\cHplus$ defined on $\mathcal{B}(\cHpluso)$ by
    \begin{align}\label{eq:kar22-affine-kernel-M}
      M(x,\D\xi)\df \frac{\langle x, \mu(\D\xi)\rangle }{\norm{\xi}^{2}},
    \end{align}
    satisfies
    \begin{align}\label{eq:kar22-affine-kernel-quasi-mono}
      \int_{\cH^+\setminus \{0\}} \langle \chi(\xi), u\rangle\,M(x,\D\xi)< \infty,  
    \end{align}
    for all $u,x\in \cH^{+}$ such that $\langle u,x \rangle = 0$;
  \item\label{item:kar22-linear-operator} an operator $B\in
    \mathcal{L}(\mathcal{H})$ with adjoint $B^{*}$ satisfying
    \begin{align*}
      \left\langle B^{*}(u) , x \right\rangle 
      - 
      \int_{\cHpluso}
      \langle \chi(\xi),u\rangle 
      \,\frac{\langle \dmu, x \rangle}{\| \xi\|^2 }
      \geq 0,
    \end{align*}
    for all $x,u\in\cHplus$ such that $\langle u,x\rangle=0$. 
  \end{defenum}
\end{definition}

\begin{remark}
  We refer to \cite{BDS55} for a general introduction to vector-valued measure
  and integration theory.
  In Corollary~\ref{coro:kar22-projected-finited-variation} and
  Remark~\ref{rem:kar22-pettis-centering} below we give a more detailed
  explanation of the quite remarkable integrability conditions in part b)
  of~\cref{item:kar22-m-2moment}
  and~\eqref{eq:kar22-affine-kernel-quasi-mono}. In particular, we explain the
  differences in the infinite-dimensional setting compared to the
  matrix-valued case and we draw a connection to the \emph{Pettis integrability}
  of the truncation function $\chi$, see also~\cite[Remark
  2.4]{CKK22a}. 
\end{remark}

Next, given an admissible parameter set $(b,B,m,\mu)$ we define the two functions
$F\colon \cHplus\to \MR$ and $R\colon \cHplus \to \cH$ as follows:
\begin{align}
  F(u)&\df \langle b,u\rangle-\int_{\cHpluso}\big(\E^{-\langle
        \xi,u\rangle}-1+\langle \chi(\xi),u\rangle\big)\,\dm,\,\quad u\in\cHplus,\label{eq:kar22-F}\\ 
  R(u)&\df B^{*}(u)-\int_{\cHpluso}\big(\E^{-\langle \xi,u \rangle}-1+\langle \chi(\xi),
        u\rangle\big)\frac{\dmu}{\norm{\xi}^{2}},\quad u\in\cHplus.\label{eq:kar22-R}
\end{align}

We recall from~\cite[Section 2 and 3]{CKK22a} that $F$ and $R$ are
well-defined and locally Lipschitz continuous on $\cHplus$. The relevance of
$F$ and $R$ lies in the fact that they determine the evolution of the functions
$\phi$ and $\psi$, which in turn control the Laplace transform of $X$ by means
of formula~\eqref{eq:kar22-affine}. We recall the \emph{generalized Riccati equations}:  
\begin{subequations}\label{eq:kar22-Riccati}
  \begin{empheq}[left=\empheqlbrace]{align}
    \,\frac{\partial \phi(t,u)}{\partial t}&=F(\psi(t,u)),\text{ for }t>0,\quad 
    \phi(0,u)=0,\label{eq:kar22-Riccati-phi}
    \\
    \,\frac{\partial \psi(t,u)}{\partial t}&=R(\psi(t,u)),\text{ for }t>0,\quad
    \psi(0,u)=u.\label{eq:kar22-Riccati-psi}
  \end{empheq}
\end{subequations}
We already proved in~\cite[Proposition 3.7]{CKK22a} that for every $u\in\cHplus$
there exists a unique continuously differentiable \emph{global solution} $(\phi(\cdot,u),\psi(\cdot,u))$
to~\eqref{eq:kar22-Riccati-phi}-\eqref{eq:kar22-Riccati-psi}, i.e. a unique
solution such that $\phi(\cdot,u)\in C^{1}(\MRplus,\MRplus)$ and
$\psi(\cdot,u)\in C^{1}(\MRplus,\cHplus)$ and we also demonstrated in~\cite[Theorem
2.3]{CKK22a} that for every admissible parameter set $(b,B,m,\mu)$ there exists
an affine Markov process $X$ on $\cHplus$ such that~\eqref{eq:kar22-affine} holds.

\section{Main results}\label{sec:kar22-main-results}
Let $(H,\langle\cdot,\cdot\rangle_{H})$ and $(\cH,\langle\cdot,\cdot\rangle)$
be as in Section~\ref{sec:kar22-admiss-param-gener} and let
$(\cH_{d},\bP_{d})_{d\in\MN}$ be a projection scheme in $\cH$ with respect to
some orthonormal basis $\set{\be_{i,j}}_{i\leq j\in\MN}$ of $\cH$ as in
Section~\ref{sec:kar22-proj-schem-hilb}. In the following three Sections~\ref{sec:galerk-appr-gener},~\ref{sec:finite-rank-operator} and \ref{sec:exist-weak-conv} we give
comprehensive versions of our main results described in the
introduction. Moreover, in Section~\ref{sec:kar22-regul-affine-proc} we
present an example of an affine process on positive Hilbert-Schmidt operators
with state-dependent jumps of infinite-variation.      

\subsection{Galerkin approximation of the generalized Riccati equations}\label{sec:galerk-appr-gener}
Recall the two functions $F$ and $R$ from equations~\eqref{eq:kar22-F}
and~\eqref{eq:kar22-R}. Then for every $d\in\MN$, we define the functions $R_{d}\colon\cHplus_{d}\to \cH_{d}$ and
$F_{d}\colon\cHplus_{d}\to \MR$ as $R_{d}(u_{d})\df\bP_{d}(R(u_{d}))$ and
$F_{d}(u_{d})\df F(u_{d})$ for $u_{d}\in\cH_{d}^{+}$. In particular, for
every $u\in\cHplus$ we have $F_{d}(\bP_{d}(u))=F(\bP_{d}(u))$ and
$R_{d}(\bP_{d}(u))=\bP_{d}(R(\bP_{d}(u)))$.\medskip{}

In the following proposition, we introduce the Galerkin-type approximation of the
operator-valued generalized Riccati equations~\eqref{eq:kar22-Riccati-phi}-\eqref{eq:kar22-Riccati-psi},
with respect to the projectional scheme $(\cH_{d},\bP_{d})_{d\in\MN}$. In particular,
we assert the existence and well-posedness of the sequence of
Galerkin approximations $\big(\phi_{d}(\cdot,\bP_{d}(u)),\psi_{d}(\cdot,\bP_{d}(u))\big)_{d\in\MN}$,
with respect to the projectional scheme $(\cH_{d},\bP_{d})_{d\in\MN}$, and we establish explicit
convergence rates for its convergence to the unique solution
$(\phi(\cdot,u),\psi(\cdot,u))$, that hold pointwise in $u\in\cHplus$ and uniformly on
compact time intervals.

\begin{proposition}\label{prop:kar22-existence-Riccati}
  Let $(b,B,m,\mu)$ be an admissible parameter set as in
  Definition~\ref{def:kar22-admissible} and for every $u\in\cHplus$ denote by
  $(\phi(\cdot,u),\psi(\cdot,u))$ the unique solution of~\eqref{eq:kar22-Riccati}.  
  Then for every $d\in\MN$, $T>0$ and $u\in\cHplus$ there exists a unique solution
  $\big(\phi_{d}(\cdot,\bP_{d}(u)),\psi_{d}(\cdot,\bP_{d}(u))\big)$ of  
  \begin{subequations}
    \begin{empheq}[left=\empheqlbrace]{align}
      \,\frac{\partial \phi_{d}(t,\bP_{d}(u))}{\partial
        t}&=F_{d}\big(\psi_{d}(t,\bP_{d}(u))\big),\quad
      \phi_{d}\big(0,\bP_{d}(u)\big)=0,\label{eq:kar22-Riccati-Galerkin-phi} \\
      \,\frac{\partial \psi_{d}(t,\bP_{d}(u))}{\partial
        t}&=R_{d}\big(\psi_{d}(t,\bP_{d}(u))\big),\quad
      \psi_{d}\big(0,\bP_{d}(u)\big)=\bP_{d}(u),\label{eq:kar22-Riccati-Galerkin-psi}
    \end{empheq}
  \end{subequations}
  such that $\phi_{d}(\cdot,\bP_{d}(u))\in C^{1}(\MRplus,\MRplus)$ and
  $\psi_{d}(\cdot,\bP_{d}(u))\in C^{1}(\MRplus,\cHplus_{d})$. Moreover, there
  exists a constant $K\geq 0$, independent of $d\in\MN$, such that   
  \begin{align}\label{eq:kar22-convergence-Galerkin}
    \sup_{t\in[0,T]}\big(|\phi_{d}(t,\bP_{d}(u))-\phi(t,u)|+\norm{\psi_{d}(t,\bP_{d}(u))-\psi(t,u)}\big)\leq
    K C_{T,d},  
  \end{align}
  where $C_{T,d}$ is given by
  \begin{align}\label{eq:kar22-convergence-Galerkin-CTd}
    C_{T,d}=\sup_{t\in[0,T]}\big(\norm{\bP^{\perp}_{d}(\E^{tB^{*}}u)}+\norm{\bP^{\perp}_{d}(\E^{tB^{*}}\mu(\cHpluso))}\big).  
  \end{align}
  In particular, for every $u\in\cHplus$ the sequence
  $\big(\phi_{d}(t,\bP_{d}(u)),\psi_{d}(t,\bP_{d}(u))\big)_{d\in\MN}$
  converges to $(\phi(t,u),\psi(t,u))$ uniformly on compact sets in time.    
\end{proposition}

\begin{definition}\label{def:kar22-solution-Riccati}
  Let $(\phi(\cdot,u),\psi(\cdot,u))$ denote the unique solution
  of~\eqref{eq:kar22-Riccati}. Then, for every $d\in\MN$, we call the function
  $\big(\phi_{d}(\cdot,\bP_{d}(u)),\psi_{d}(\cdot,\bP_{d}(u))\big)$ in
  Proposition~\ref{prop:kar22-existence-Riccati} the $d^{{\scriptstyle
      th}}$-\emph{Galerkin approximation}  of $(\phi(\cdot,u),\psi(\cdot,u))$.  
\end{definition}

\subsection{Finite-rank operator-valued affine processes}\label{sec:finite-rank-operator}

For every $d\in\MN$, we define the set $\cD^{d}\df\set{\E^{-\langle
    \cdot,u\rangle}\colon\, u\in\cHplus_{d}}\subseteq C(\cHplus,\MR)$ and the
operator $\cG^{d}\colon \cD\to C(\cHplus,\MR)$ as  
\begin{align}\label{eq:kar22-G-d-operator}
  \cG^{d}\E^{-\langle
      \,\cdot\,,u\rangle}(x)\df\big(-F_{d}(u)-\langle x
  ,R_{d}(u)\rangle\big)\E^{-\langle x, u\rangle},\quad x\in\cHplus,
\end{align}
where $F_{d}$ and $R_{d}$ are defined as in
Section~\ref{sec:galerk-appr-gener} above. The following proposition asserts,
that for all $d\in\MN$, the $d^{{\scriptstyle th}}$-Galerkin approximation
$\big(\phi_{d}(\cdot,\bP_{d}(u)),\psi_{d}(\cdot,\bP_{d}(u))\big)$ gives rise
to an affine Markov process $X^{d}$ with values in $\cHplus_{d}$, that solves
the martingale problem for $\cG^{d}$ with $X^{d}_{0}=\bP_{d}(x)$ on a suitable
stochastic basis.

\begin{proposition}\label{prop:kar22-embedding-affine-main}
  Let the assumptions of Proposition~\ref{prop:kar22-existence-Riccati}
  hold. Then for every $d\in\MN$ the following holds true:
  \begin{propenum}
  \item[i)]\label{item:kar22-embedding-affine-main-1} There exists a unique Markov
    process $(X^{d},(\MP_{x}^{d})_{x\in\cHplus})$, realized on the space
    $D(\MRplus,\cHplus)$ of all c\`adl\`ag paths and where $\MP_{x}^{d}$
    denotes the law of $X^{d}$ given $X_{0}^{d}=\bP_{d}(x)$, such that for all
    $x\in\cHplus$ we have $\MP_{x}^{d}(\set{X^{d}_{t}\in\cHplus_{d}\colon\, t\geq 0})=1$ and the
    following affine transform formula holds true:    
    \begin{align}\label{eq:kar22-affine-Galerkin}
      \hspace{10mm} \EXspec{\mathbb{P}^{d}_{x}}{\E^{-\langle X^{d}_{t},
      \bP_{d}(u)\rangle}}=\E^{-\phi_{d}(t,\bP_{d}(u))-\langle
      \bP_{d}(x),\psi_{d}(t,\bP_{d}(u))\rangle},\,\, t\geq 0,\,u\in\cHplus,  
    \end{align}
    for $\big(\phi_{d}(\cdot,\bP_{d}(u)),\psi_{d}(\cdot,\bP_{d}(u))\big)$ the unique solution of~\eqref{eq:kar22-Riccati-Galerkin-phi}-\eqref{eq:kar22-Riccati-Galerkin-psi}.
  \item[ii)]\label{item:kar22-embedding-affine-main-2} For every $x\in\cHplus$
    and every $u\in\cHplus$ the process 
    \begin{align}\label{eq:kar22-Gd-martingale-problem}
      \Big(\E^{-\langle X_{t}^{d},\bP_{d}(u)\rangle}-\E^{-\langle
      \bP_{d}(x),\bP_{d}(u)\rangle}-\int_{0}^{t}(\cG^{d}\E^{-\langle
      \,\cdot\,,\bP_{d}(u)\rangle})(X_{s}^{d})\,\D s\Big)_{t\geq 0},
    \end{align}
    is a real-valued martingale with respect to the stochastic basis
    $(\Omega,\bar{\cF}^{d},\bar{\MF}^{d},\!\MP_{x}^{d})$, where
    $\Omega=D(\MRplus,\cHplus)$ and $\bar{\MF}^{d}=(\bar{\cF}^{d}_{t})_{t\geq
      0}$ denotes the augmentation of
    the natural filtration $(\cF^{d}_{t})_{t\geq 0}$ of $X^{d}$ with respect to
    the measure $\MP_{x}^{d}$ from i).
  \end{propenum}
\end{proposition}
The proof of Proposition~\ref{prop:kar22-embedding-affine-main} uses results
from the matrix-valued case in~\cite{CFMT11} and a subsequent transform into the
set of finite-rank operators. In addition to Proposition~\ref{prop:kar22-embedding-affine-main}, we show some
additional properties of the finite-rank operator-valued processes
$(X^{d})_{d\in\MN}$ in Section~\ref{sec:kar22-affine-finite-rank} below, namely:
We present a semimartingale representation of $(X_{t}^{d})_{t\geq 0}$ in
Proposition~\ref{prop:kar22-embedding-affine}, give a more detailed description
of the operator $\cG^{d}$ in Proposition~\ref{prop:martingale-problem-d} and
show that all the processes $(X^{d})_{d\in\MN}$ must be of finite-variation in Corollary~\ref{coro:kar22-projected-finited-variation}.

\subsection{Existence and weak convergence}\label{sec:exist-weak-conv}
In the following, we follow the setting in~\cite{Tem71}, namely: We let $(H,\langle
\cdot,\cdot\rangle_{H})$ be a separable Hilbert space and assume that
$(V,(\cdot,\cdot)_{V})$ is a second separable Hilbert space such that
$V\subseteq H$ and assume:

\vspace{2mm}
\begin{assumption}{}{C}\label{assump:kar22-compact-embedding}
  $(V,(\cdot,\cdot)_{V})$ is densely and compactly embedded in
  $(H,\langle\cdot,\cdot\rangle_{H})$. 
\end{assumption}
\vspace{2mm}
Moreover, let us denote by $V^{*}$ the Hilbert space dual of $V$ with respect
to the inner-product $\langle \cdot,\cdot\rangle_{H}$. Then identifying $H$
with its dual space $H^{*}$ gives the \emph{Gelfand triple}: $V\hookrightarrow
H \hookrightarrow V^{*}$. Moreover, we define the space $\cV$ as  
\begin{align}\label{eq:kar22-intersection-HS-space}
  \cV\df \cL_{2}(V^{*},H)\cap \cL_{2}(H,V), 
\end{align}
equip $\cV$ with the inner-product $\langle
\cdot,\cdot\rangle_{\cV}\df\langle\cdot,\cdot
\rangle_{\cL_{2}(V^{*},H)}+\langle\cdot, \cdot\rangle_{\cL_{2}(H,V)}$ and
denote the induced norm by $\norm{\cdot}_{\cV}$. Note that $\cV\subseteq \cL_{2}(H)$
and $(\cV,\langle\cdot,\cdot\rangle_{\cV})$ is a Hilbert space which is itself
densely and compactly embedded in $(\cL_{2}(H),\langle\cdot,\cdot\rangle)$,
see~\cite{Tem71}. Lastly, we define the space $\cV_{0}\subseteq
\cV$ as the subspace of all self-adjoint operators (with respect to
$\langle\cdot,\cdot\rangle$), i.e. $\cV_{0}\df \cV\cap \cH$.\par{}

We then proceed with our main result on the existence and approximation of affine
processes on positive Hilbert-Schmidt operators. In addition, we assert the
existence of c\`adl\`ag versions and give a semimartingale description of this
affine class. 

\begin{theorem}\label{thm:kar22-main-convergence}
  Let $(b,B,m,\mu)$ be an admissible parameter set as in Definition~\ref{def:kar22-admissible} and let
  Assumption~\ref{assump:kar22-compact-embedding} be satisfied. Then the
  following holds true: 
  \begin{theoremenum}
  \item\label{item:kar22-main-convergence-1} There exists a unique affine process
    $(X,(\MP_{x})_{x\in\cHplus})$ on $\cHplus$, with paths in
    $D(\MRplus,\cHplus)$ and where $\MP_{x}$ denotes the law of $X$ given
    $X_{0}=x$, such that for every $x\in\cHplus$ we have  
    \begin{align}\label{eq:kar22-affine-Galerkin-theorem}
      \EXspec{\MP_{x}}{\E^{-\langle X_{t}, u\rangle}}=\E^{-\phi(t,u)-\langle
      x,\psi(t,u)\rangle},\quad t\geq 0,\, u\in\cHplus,  
    \end{align}
    for $(\phi(\cdot,u),\psi(\cdot,u))$ the unique solution
    of~\eqref{eq:kar22-Riccati-phi}-\eqref{eq:kar22-Riccati-psi}.
  \item\label{item:kar22-main-convergence-2} Moreover, let $(X^{d})_{d\in\MN}$
    be the sequence of finite-rank operator-valued affine processes in Proposition~\ref{prop:kar22-embedding-affine-main}. Then
    the sequence $(X^{d})_{d\in\MN}$ converges weakly to $X$ on $D(\MRplus,\cHplus)$
    equipped with the Skorohod topology, i.e. for all $f\in
    C(D(\MRplus,\cHplus),\MR)$  we have
    \begin{align*}
      \EXspec{\MP_{x}^{d}}{f(X^{d})}\to \EXspec{\MP_{x}}{f(X)},\quad\text{as
      }d\to\infty. 
    \end{align*}
    If, in addition, we have $\norm{\mu(\cHpluso)}_{\cV}<\infty$ and
    $B^{*}(\cV_{0})\subseteq \cV_{0}$, then for every $T>0$ and $u\in\cHplus$
    with $\norm{u}_{\cV}\leq 1$ there exists a constant $C_{T}>0$ such that
    for all $d\in\MN$: 
    \begin{align}\label{eq:kar22-main-convergence-rate}
      \hspace{10mm}\sup_{t\in [0,T]}
      \left\lvert\EXspec{\MP_{x}}{\E^{-\langle u,
      X_{t}\rangle}}-\EXspec{\MP^{d}_{x}}{\E^{-\langle u,
      X^{d}_{t}\rangle}}\right\rvert&\leq C_{T}\norm{\bP^{\perp}_{d}}_{\cL(\cV,\cH)}(1+\norm{x}).
    \end{align}
  \end{theoremenum}
\end{theorem}

\begin{remark}\label{rem:kar22-main-convergence}
  \begin{enumerate}
  \item[i)] A more detailed version of the constant $C_{T}$
    showing up in~\eqref{eq:kar22-main-convergence-rate} can be found in Corollary~\ref{cor:kar22-convergence-Galerkin-compact} below.    
  \item[ii)] Analogous to~\cref{item:kar22-embedding-affine-main-2} in the
    finite-rank case, we will
    show in Proposition~\ref{prop:kar22-weak-convergence} that $X$ is the
    (unique) solution to the martingale problem for the operator $\cG$ given by $\cG\E^{-\langle \,\cdot\, u\rangle}=(-F(u)-\langle\,\cdot\,,
    R(u)\rangle)\E^{-\langle \cdot,u \rangle}$ on the set
    $\cD\df\set{\E^{-\langle\cdot,u\rangle}\colon u\in\cHplus}$. To keep this
    section reasonably concise we relegate this (and other) side results to the
    latter Sections~\ref{sec:kar22-affine-finite-rank}
    and~\ref{sec:kar22-tightness-weak-convergence}.  
  \end{enumerate}
\end{remark}

The following proposition asserts that the affine process $(X,\MP_{x})$
from~\cref{item:kar22-main-convergence-1} is a semimartingale and we specify its
semimartingale characteristics, see, e.g.,~\cite{CKK22b}:

\begin{proposition}\label{prop:kar22-main-convergence-3}
  For every $x\in\cHplus$ the
    process $X$ is a square-integrable semimartingale with respect to the
    stochastic basis $(\Omega,\bar{\cF},\bar{\MF},\MP_{x})$, 
    where $\bar{\MF}=(\bar{\cF}_{t})_{t\geq 0}$ denotes the augmentation of
    the natural filtration $(\cF_{t})_{t\geq 0}$ of $X$ with respect to the
    measure $\MP_{x}$. The semimartingale characteristics $(A,C,\nu)$ of $X$,
    with respect to $\chi$, are given by: 
    \begin{align}
      A_{t}&=\int_{0}^{t}\big(b+B(X_{s})\big)\D s,\quad t \geq 0,\label{eq:kar22-characteristic-A}\\
      C_{t}&=0,\quad t \geq 0,\label{eq:kar22-characteristic-C}\\
      \nu(\D t,\D\xi)&= \big(m(\D\xi)+M(X_{t},\D\xi)\big)\D t,\label{eq:kar22-characteristic-nu}
    \end{align}
    and for every $t\geq 0$ the following representation holds true
    \begin{align}
      X_{t}&= x + \int_0^t \big( b+ B(X_{s})+ \int_{\cHplus \cap \{\|\xi\|
             >1\}} \xi\,\nu(X_s, \D \xi))\big)\, \D s + \bar{J}_{t},
    \end{align}
    where $\nu(x,\D\xi)=m(\D\xi)+M(x,\D\xi)$ and $(\bar{J}_{t})_{t\geq 0}$ is a purely
    discontinuous square-integrable martingale of the $\nu(\D t,\D\xi)$-compensated jumps of $X$ .

\end{proposition}
\begin{proof}
  Follows immediately from~\cite[Proposition 2.4]{CKK22b}.  
\end{proof}

\subsection{Examples}\label{sec:kar22-regul-affine-proc}
In this section we give two examples of affine processes on positive
Hilbert-Schmidt operators: The first one is simple, yet of infinite
variation. The second example is a generic, but very high parametric.

\subsubsection{A simple affine process on positive Hilbert-Schmidt operators
  with jumps of infinite variation}\label{sec:simple-affine-proc}
We set $b=0$ and $m(\D\xi)=0$, i.e. we assume that the constant drift and
constant jump coefficients vanish. Then, for every $n\in\MN$ define the measures $\mu_{n}=n^{-2}\delta_{n^{-1}}$, where we note that
\begin{align*}
  \mu_{n}(\set{\lambda\in (0,\infty)\colon\,\lambda \be_{n}\in
  A})=\frac{1}{n^{2}}\delta_{n^{-1} e_{n}\otimes e_{n}}(A),\quad A\in\cB(\cHpluso).
\end{align*}
Next, let $g\in\cHplus$ be arbitrary and define
\begin{align*}
  \mu(A)\df\sum_{n\in\MN}g\mu_{n}(\set{\lambda\in (0,\infty)\colon \lambda
  \be_{n}\in A}),\quad A\in\cB(\cHpluso).
\end{align*}

Then set $B(u)\df
\big(\sum_{n\in\MN}\big(\!\int_{0}^{1}\!\lambda^{-1}\mu_{n}(\D\lambda)\big)
\langle u,\be_{n}\rangle\big)\langle g,x\rangle$, which ensures
that~\cref{item:kar22-linear-operator} is satisfied and hence $(0,B,0,\mu)$ is
admissible. This, by Theorem~\ref{thm:kar22-main-convergence}, implies that there exists a unique
associated affine process $X$. Next, we show that with this choice of $\mu$, the measure $M(x,\D\xi)$
in~\eqref{eq:kar22-affine-kernel-M} is of infinite-variation. Indeed, we see that 
\begin{align*}
  \int_{0<\norm{\xi}\leq 1}\norm{\xi}M(x,\D\xi)&=  \int_{0<\norm{\xi}\leq
                                                 1}\frac{1}{\norm{\xi}}\sum_{n\in\MN}\langle
                                                 x, g\rangle
                                                 \delta_{n^{-1}e_{n}\otimes e_{n}}(\D\xi)\\
                                               &=\sum_{n\in\MN}\frac{1}{\norm{n^{-1}e_{n}\otimes e_{n}}}\frac{1}{n^{2}}\langle
                                                 x, g\rangle \\
                                               &=\langle x, g\rangle
                                                 \sum_{n\in\MN}\frac{1}{n}=\infty. 
\end{align*}
This then implies that also the measure $\nu$ in~\eqref{eq:kar22-characteristic-nu}
is of infinite-variation, which in turn means that the associated affine process
$X$ has infinite-variation. However, note that for every $u\in\cH$, we have
\begin{align*}
  \int_{0<\norm{\xi}\leq 1}\langle \xi, u\rangle
  M(x,\D\xi)&=\sum_{n\in\MN}\langle n^{-1}e_{n}\otimes e_{n},
              u\rangle\frac{1}{\norm{n^{-1}e_{n}\otimes
              e_{n}}^{2}}\frac{1}{n^{2}}\langle x, g\rangle \\
            &=\langle x, g\rangle \sum_{n\in\MN}\frac{1}{n}u_{n,n} <\infty,
\end{align*}
with $u_{n,n}=\langle e_{n}\otimes e_{n}, u\rangle$ and $(n^{-1})_{n\in\MN}\in\ell^{2}$
as well as $(u_{n,n})_{n\in\MN}\in \ell^{2}$, hence
$(n^{-1}u_{n,n})_{n\in\MN}\in\ell^{1}$. To see that $(u_{n,n})_{n\in\MN}\in
\ell^{2}$, note that $u\in \cH$ and hence 
\begin{align*}
  \norm{(u_{n,n})_{n\in\MN}}_{\ell^{2}}=\big(\sum_{n\in\MN}u_{n,n}^{2}
  \big)^{1/2}\leq \big(\sum_{n,m\in\MN}u_{n,m}^{2}\big)^{1/2}=\norm{u}<\infty.  
\end{align*}

This demonstrates that although $X$ has infinite-variation, it is necessarily of
finite-variation in every fixed direction $u\in\cH$. A phenomenon that is not
possible for affine processes on finite-dimensional state spaces.

\subsubsection{A generic affine pure-jump process on
  \texorpdfstring{$\cHplus$}{the cone of positive self-adjoint Hilbert-Schmidt operators}}
In this section we present an example of an affine process on positive
Hilbert-Schmidt operators with c\`adl\`ag paths and jumps of
infinite-variation. Moreover, we describe its finite-rank approximations
through a semimartingale representation. As before we let $\set{\be_{i,j}}_{i\leq j\in
  \MN}$ denote an orthonormal basis of $\cH$ and let $(\cH_{d},\bP_{d})_{d\in\MN}$ be the
associated finite-rank projectional scheme. We specify the parameter set $(b,B,m,\mu)$ as follows:  
\begin{enumerate}
\item[i)] We set $m(\D\xi)=\norm{\xi}^{-2}\eta(\D\xi)$ for $\eta\colon
  \cB(\cHplus\setminus\set{0})\to [0,\infty]$ given by
  \begin{align}\label{eq:kar22-example-m}
    \eta(A)=\sum_{n\in\MN}\eta_{n}(\set{\lambda\in (0,\infty)\colon \lambda
    \be_{n}\in A}),\quad A\in\cB(\cHpluso),
  \end{align}
  where $(\eta_{n})_{n\in\MN}$ is a sequence of finite measures on
  $\cB((0,\infty))$ such that for all $u\in\cH$ we have 
  \begin{align}\label{eq:kar22-example-jump-cond-m}
    \sum_{n\in\MN}\big(\int_{0}^{1}\lambda^{-1}
    \eta_{n}(\D\lambda)\big)\langle u, \be_{n}\rangle <\infty
    \quad\text{and}\quad \sum_{n\in\MN}\eta_{n}((0,\infty))<\infty.     
  \end{align}
\item[ii)] We let $\tilde{b}\in\cHplus$ be arbitrary and let $I_{m}\in\cH$ be
  such that for all $u\in\cH$ we have
  \begin{align*}
    \langle
    I_{m},u\rangle=\sum_{n\in\MN}\big(\int_{0}^{1}\lambda^{-1}\eta_{n}(\D\lambda)\big)\langle
    u, \be_{n}\rangle,  
  \end{align*}
  and define $b\df \tilde{b}+I_{m}$.
\item[iii)] We let $(g_{n})_{n\in\MN}\subseteq\cHplus$ and define
  $\mu(\D\xi)\colon \cB(\cHpluso)\to \cHplus$ by    
  \begin{align}\label{eq:kar22-example-mu}
    \mu(A)=\sum_{n\in\MN}g_{n}\mu_{n}(\set{\lambda\in (0,\infty)\colon \lambda \be_{n}\in A}), 
  \end{align}
  where $(\mu_{n})_{n\in\MN}$ is a sequence of finite-measures on
  $\cB((0,\infty))$ such that for all $x\in\cHplus$ and $u\in\cH$ we have 
  \begin{align}\label{eq:kar22-example-jump-cond-mu}
   \quad\qquad\sum_{n\in\MN}\big(\int_{0}^{1}\lambda^{-1}\mu_{n}(\D\lambda)\big)\langle
    g_{n},x\rangle \langle u,\be_{n}\rangle<\infty \quad\text{and}\quad
    \sum_{n\in\MN}g_{n}\mu_{n}((0,\infty))\in \cHplus.   
  \end{align}
  Moreover, for every $x\in\cHplus$ we set $M(x,\D\xi)\df \norm{\xi}^{-2}\langle x,\mu(\D\xi)\rangle$. 
\item[iv)] Finally, let $C$ be a bounded linear operator on $H$ and let
  $\Gamma\in\cL(\cH)$ be such that for all $u,x\in\cH$ we have 
  \begin{align*}
    \langle \Gamma(x),
    u\rangle=\sum_{n\in\MN}\big(\int_{0}^{1}\lambda^{-1}\mu_{n}(\D\lambda)\big)\langle
    g_{n},x\rangle \langle u,\be_{n}\rangle     
  \end{align*}
  Then we define $B\in\cL(\cH)$ by $B(u)\df Cu+uC^{*}+\Gamma(u)$.  
\end{enumerate}

It can be seen that the parameter set $(b,B,m,\mu)$ is correctly set up to
satisfy the conditions of Definition~\ref{def:kar22-admissible}. Moreover, for every $d\in\MN$ and $x\in\cHplus_{d}$ we set 
\begin{enumerate}
\item[i)] $m_{d}(A)\df\sum\limits_{n=1}^{d}m_{n}(\set{\lambda\in (0,\infty)\colon \lambda
    \be_{n}\in A})$ for $A\in \cB(\cHplus_{d}\setminus\set{0})$;
\item[ii)]$c_{d}\df \tilde{b}_{d}+\sum\limits_{n=1}^{d}\big(\int_{0}^{1}\lambda^{-1}
  \eta_{n}(\D\lambda)\big)\be_{n}$ for $\tilde{b}_{d}=\bP_{d}(\tilde{b})$;
\item[iii)] $\mu_{d}(A)\df\sum\limits_{n=1}^{d}\bP_{d}(g_{n})\mu_{n}(\set{\lambda\in
    (0,\infty)\colon \lambda \be_{n}\in A})$ for $A\in \cB(\cHplus_{d}\setminus\set{0})$;
\item[iv)] $D_{d}(x)\df C_{d}x+xC^{*}_{d}+\sum_{n=1}^{d}\big(\!\int_{0}^{1}\!\lambda^{-1}\mu_{n}(\D\lambda)\big)\langle
    g_{n},x\rangle \langle u,\be_{n}\rangle$ for
  $C\df \sP_{d}C$. 
\end{enumerate}
It follows from Definition~\ref{def:kar22-admissible-Galerkin} and
Proposition~\ref{prop:kar22-embedding-affine} below, that there exists an
affine process $X^{d}=(X_{t}^{d})_{t\geq 0}$ with values in $\cHplus_{d}$ admitting
the following representation:    
\begin{align}\label{eq:kar22-example-semimartingale-representation-1}
  X^{d}_{t}&= X^{d}_0 +
             \int_0^t\Big(\tilde{b}_{d}+C_{d}X^{d}_{s}+X^{d}_{s}C_{d}^{*}\Big)\,\D
             s+\int_{0}^{t}\int_{\cHplus_{d}}\xi \,\mu^{X^{d}}(\D
             t,\D\xi),\quad t \geq 0, 
\end{align}
where $\mu^{X^{d}}(\D t,\D\xi)$ denotes the random measure associated with the
jumps of $X^{d}$ with compensator $\nu^{X^{d}}(\D
t,\D\xi)=(m_{d}(\D\xi)+\norm{\xi}^{-2}\langle X_{t}^{d},\mu(\D\xi)\rangle)\D
t$. From this it can be inferred that the process $X^{d}$ must be the
affine process from~\cref{item:kar22-embedding-affine-main-2}. It thus follows
from Theorem~\ref{thm:kar22-main-convergence}, that whenever $H$ satisfies
Assumption~\ref{assump:kar22-compact-embedding}, there exists a unique affine
process $X=(X_{t})_{t\geq 0}$ on $\cHplus$ that can be represented as 
\begin{align}\label{eq:kar22-example-semimartingale-representation}
  X_{t}&= X_0 + \int_0^t \Big(\tilde{b}+I_{m}+CX_{s}+X_{s}C^{*}+\Gamma(X_{s})\Big)\D s\nonumber\\
       &\qquad+\int_{0}^{t}\Big(\int_{\cHplus \cap \{\|\xi\| >1\}} \xi
         \,(m(\D\xi)+M(X_s, \D \xi)\Big)\, \D s + \bar{J}_{t},\quad t\geq 0, 
\end{align}
where $M(x,\D\xi)$ is as in~\eqref{eq:kar22-affine-kernel-M} and
$(\bar{J}_{t})_{t\geq 0}$ is a purely-discontinuous square-integrable
martingale. Moreover, it follows that $X^{d}$ converges weakly to $X$ as
$d\to\infty$.   
Note that in~\eqref{eq:kar22-example-semimartingale-representation-1} the
$\cH_{d}$-projections of the two drift terms $I_{m}$ and $\Gamma$ are killed by the compensator of the jump-process
$\big(\int_{0}^{t}\int_{\cHplus_{d}}\xi\,\mu^{X^{d}}(\D t,\D\xi)\big)_{t\geq
  0}$, as for every $d\in\MN$ the jumps are of finite-variation. In the
limit case~\eqref{eq:kar22-example-semimartingale-representation}, however,
$I_{m}$ and $\Gamma$ must occur in the drift (and jump-part) again as the
driving jump-process possibly converges to a process of infinite-variation as we
saw in Example~\ref{sec:simple-affine-proc} before.

\section{Proof: Galerkin approximation of generalized Riccati equations}\label{sec:proof-prop-refpr}

This section is devoted to the proof of
Proposition~\ref{prop:kar22-existence-Riccati}. If, in addition,
Assumption~\ref{assump:kar22-compact-embedding} is satisfied, then
Corollary~\ref{cor:kar22-convergence-Galerkin-compact} below sharpens the
convergence rate~\eqref{eq:kar22-convergence-Galerkin-CTd} even further. We begin this
section with a short lemma on the local Lipschitz continuity of the functions
$F$, $R$, $(F_{d})_{d\in\MN}$ and $(R_{d})_{d\in\MN}$.

\begin{lemma}\label{lem:kar22-locally-uniform-convergence-Rd}
  Let $M>0$ and $d\in\MN$. Then for every $u,v\in\cHplus_{d}$ with $\norm{u}\vee\norm{v}\leq M$ we have
  \begin{align}
    |F_{d}(u)-F_{d}(v)|\leq \Big(\norm{b}+(M+1)\int_{\cHpluso}\norm{\xi}^{2}\,m(\D\xi)\Big)\norm{u-v},\label{eq:kar22-Lipschitz-Fd}\\
    \norm{R_{d}(u)-R_{d}(v)}\leq
    \Big(\norm{B}_{\cL(\cH)}+(M+1)\norm{\mu(\cHpluso)}\Big)\norm{u-v}.\label{eq:kar22-Lipschitz-Rd}  
  \end{align}
  Moreover, for every $u,v\in\cHplus$ with $\norm{u}\vee\norm{v}\leq M$ we can replace
  $F_{d}$ by $F$ and $R_{d}$ by $R$, respectively, and the
  inequalities~\eqref{eq:kar22-Lipschitz-Fd} and~\eqref{eq:kar22-Lipschitz-Rd}
  continue to hold with the same local Lipschitz constants.
\end{lemma}
\begin{proof}
  We prove the inequalities for $F$ and $R$ first. Let $M>0$ and
  $u,v\in\cHplus$ such that $\norm{u}\vee\norm{v}\leq M$ and note that for all
  $\xi\in\cHplus$ we have 
  \begin{align*}
    \lvert \E^{-\langle\xi,u\rangle}-\E^{-\langle\xi,v\rangle}+\langle
    \xi,u-v\rangle\rvert\leq \norm{\xi}^{2}(\norm{u}\vee
    \norm{v})\norm{u-v},
  \end{align*}
  and $|\E^{-\langle\xi,u\rangle}-\E^{-\langle\xi,v\rangle}|\leq |\langle\xi,
  u-v\rangle|$, see also~\cite[Remark 3.4]{CKK22a}, we thus see that 
  \begin{align}\label{eq:kar22-Lipschitz-F}
    |F(u)-F(v)|&\leq |\langle b,
    u-v\rangle|+\int_{\cHpluso}|\E^{-\langle\xi,u\rangle}-\E^{-\langle\xi,u\rangle}+\langle
                 \chi(\xi),u-v\rangle|\,m(\D\xi)\nonumber\\
               &\leq \norm{b}\norm{u-v}+M\big(\int_{\cHplus\cap\set{0<\norm{\xi}\leq
                 1}}\norm{\xi}^{2}\,m(\D\xi)\big)\norm{u-v}\nonumber\\
    &\quad +\big(\int_{\cHplus\cap\set{\norm{\xi}>1}}\norm{\xi}\,m(\D\xi)\big)\norm{u-v}\nonumber\\
               &\leq
                 \Big(\norm{b}+(M+1)\int_{\cHpluso}\norm{\xi}^{2}\,m(\D\xi)\Big)\norm{u-v},
  \end{align}
  which proves inequality~\eqref{eq:kar22-Lipschitz-Fd} for $F_{d}$ replaced
  with $F$. For every $d\in\MN$ it is then obvious that also the function $F_{d}$ is
  Lipschitz continuous on the set $\cHplus\cap \set{u\in\cH\colon \norm{u}\leq M}$
  with the same Lipschitz constant as $F$ in~\eqref{eq:kar22-Lipschitz-F} admits. For
  the second inequality~\eqref{eq:kar22-Lipschitz-Rd}, again at first for $R$ replacing
  $R_{d}$, we note that by the monotonicity of the cone $\cHplus$ we have 
  \begin{align}\label{eq:kar22-Lipschitz-R}
    \norm{R(u)-R(v)}&\leq
                      \norm{B^{*}}_{\cL(\cH)}\norm{u-v}+\Big\|\int_{\cHplus\cap\set{\norm{\xi}>1}}\big\lvert\E^{-\langle\xi,u\rangle}-\E^{-\langle\xi,u\rangle}\big\rvert\,\frac{\mu(\D\xi)}{\norm{\xi}^{2}}\Big\|\nonumber\\
                    &\quad+\Big\|\int_{\cHplus\cap\set{0<\norm{\xi}\leq
                      1}}\big\lvert\E^{-\langle\xi,u\rangle}-\E^{-\langle\xi,u\rangle}+\langle\xi,u-v\rangle\big\rvert\,\frac{\mu(\D\xi)}{\norm{\xi}^{2}}\Big\|\nonumber\\ 
                    &\leq
                      \left(\norm{B}_{\cL(\cH)}+(M+1)\norm{\mu(\cHpluso)}\right)\norm{u-v}. 
  \end{align}
  We see that inequality~\eqref{eq:kar22-Lipschitz-R} also
  holds for every $(R_{d})_{d\in\MN}$ with the same local Lipschitz constant
  on $\set{u\in\cHplus\colon \norm{u}\leq M}$ given by~\eqref{eq:kar22-Lipschitz-R}.
  In particular, the Lipschitz constants of $R_{d}$ and $F_{d}$ on
  $\cHplus\cap \set{u\in\cHplus_{d}\colon \norm{u}\leq M}$ do not depend on
  $d\in\MN$.     
\end{proof}

With this lemma at hand we can now prove
Proposition~\ref{prop:kar22-existence-Riccati}.  

\begin{proof}[Proof of Proposition~\ref{prop:kar22-existence-Riccati}]
  Let $u\in\cHplus$, $T>0$ and $d\in\MN$. We begin with showing the existence
  and uniqueness of the solution
  to~\eqref{eq:kar22-Riccati-Galerkin-phi}-\eqref{eq:kar22-Riccati-Galerkin-psi}
  on the interval $[0,T]$. From~\eqref{eq:kar22-affine-kernel-quasi-mono} it
  follows that for every $x,v\in\cHplus_{d}$ such that $\langle x,v\rangle=0$
  we have
  \begin{align*}
    \langle R_{d}(v), x\rangle=\langle \bP_{d}(R(v)), x\rangle=\langle R(v), x\rangle\geq 0.   
  \end{align*}
  Thus then implies that $R_{d}$ is \emph{quasi-monotone increasing} with respect to
  the cone $\cHplus_{d}$, see also~\cite[Definition 3.1]{CKK22a}. This, the
  Lipschitz continuity of $R_{d}$ on the sets $\set{v\in\cHplus_{d}\colon
    \norm{v}\leq M}$, for every $M>0$, see
  Lemma~\ref{lem:kar22-locally-uniform-convergence-Rd}, implies the existence
  and uniqueness of a continuously differentiable function
  $\psi_{d}(\cdot,\bP_{d}(u))$ on $[0,T]$ that
  solves~\eqref{eq:kar22-Riccati-Galerkin-psi}, see also the proof of~\cite[Proposition 3.5]{CKK22a}. The existence and uniqueness of a
  continuously differentiable function $\phi_{d}(\cdot,\bP_{d}(u))$ on $[0,T]$
  solving~\eqref{eq:kar22-Riccati-Galerkin-phi} then follows immediately from
  the continuity of $F_{d}$ and mere integration of both sides
  of~\eqref{eq:kar22-Riccati-Galerkin-phi}.\\
  Next, we prove the inequality~\eqref{eq:kar22-convergence-Galerkin}. For
  this let us fix $M>0$ and note that by
  Lemma~\ref{lem:kar22-locally-uniform-convergence-Rd} we find a Lipschitz
  constant of $R_{d}$ and $R$ on $\set{v\in\cHplus_{d}\colon \norm{v}\leq M}$
  which does not depend on $d\in\MN$. It thus follows from~\cite[Equation
  3.11]{CKK22a} (see also the proof of~\cite[Proposition 3.7]{CKK22a}) that
  for all $t\in [0,T]$ and all $u\in\set{u\in\cHplus\colon \norm{u}\leq M}$ we have
  \begin{align*}
    \norm{\psi(t,u)}\vee \norm{\psi_{d}(t,\bP_{d}(u))}\leq M
    \exp\big((\norm{B}_{\cL(\cH)}+2\norm{\mu(\cHpluso)})T\big).    
  \end{align*}
  Let us set $H_{M}\df M
  \exp\big((\norm{B}_{\cL(\cH)}+2\norm{\mu(\cHpluso)})T\big)$ and note that
  for every $t\in [0,T]$ and $u\in\cHplus$ we have 
  \begin{align*}
    \norm{\psi(t,u)-\psi_{d}(t,\bP_{d}(u))}&\leq
                                             \norm{\psi(t,u)-\bP_{d}(\psi(t,u))}+\norm{\bP_{d}(\psi(t,u))-\psi_{d}(t,\bP_{d}(u))},                           
  \end{align*}
  where by~\eqref{eq:kar22-Lipschitz-Rd} for all $u\in\cHplus$ with
  $\norm{u}\leq M$ the second term satisfies  
  \begin{align}\label{eq:kar22-Riccati-convergence-1}
    \norm{\bP_{d}(\psi(t,u))-\psi_{d}(t,\bP_{d}(u))}&\leq
                                                      \int_{0}^{t}\norm{\bP_{d}R(\psi(s,u))-R_{d}(\psi_{d}(s,\bP_{d}(u)))}\,\D
                                                     s\nonumber \\
                                                    &\leq L^{(1)}_{M}\int_{0}^{t}\norm{\psi(s,u)-\psi_{d}(s,\bP_{d}(u))}\,\D
                                                      s,   
  \end{align}
  with $L^{(1)}_{M}\df\norm{B}_{\cL(\cH)}+(H_{M}+1)\norm{\mu(\cHpluso)}$. Moreover, for all
  $u,\xi\in\cH$ we set $K_{u}(\xi)\df\E^{-\langle \xi,u\rangle}-1+\langle\chi(\xi),u\rangle$ and recall that by the
  variation-of-constant formula the solution $\psi(\cdot,u)$ satisfies 
  \begin{align*}
    \psi(t,u)=\E^{t
    B^{*}}u+\int_{0}^{t}\E^{(t-s)B^{*}}\Big(\int_{\cHpluso}K_{\psi(s,u)}(\xi)\frac{\mu(\D\xi)}{\norm{\xi}^{2}}\Big)\D
    s,\quad
    t\in [0,T].
  \end{align*}
  From this and writing $\norm{\psi(t,u)-\bP_{d}(\psi(t,u))}=\norm{\bP^{\perp}_{d}(\psi(t,u))}$ we obtain
  \begin{align}
    \norm{\bP^{\perp}_{d}(\psi(t,u))}&\leq
                                       \norm{\bP_{d}^{\perp}(\E^{tB^{*}}u)}+\int_{0}^{t}\norm{\bP^{\perp}_{d}\E^{(t-s)B^{*}}\int_{\cHpluso}K_{\psi(s,u)}(\xi)\frac{\mu(\D\xi)}{\norm{\xi}^{2}}}\D
                                       s\nonumber\\ 
                                       &\leq
                                         \norm{\bP^{\perp}_{d}(\E^{tB^{*}}u)}+t
                                         H_{M}^{2}\sup_{s\in[0,t]}\norm{\bP^{\perp}\E^{sB^{*}}_{d}(\mu(\cHpluso))},\label{eq:kar22-Riccati-convergence-2}  
  \end{align}
  where in the last line of~\eqref{eq:kar22-Riccati-convergence-2} we used that
  \begin{align*}
    |\E^{-\langle \xi,\psi(s,u)\rangle}\!-\!1\!+\!\langle
  \chi(\xi),\psi(s,u)\rangle|\leq
  \frac{1}{2}\norm{\xi}^{2}\norm{\psi(s,u)}^{2}\one_{\set{\norm{\xi}\leq
    1}}\!+\!\norm{\xi}\norm{\psi(s,u)}\one_{\set{\norm{\xi}>1}} 
  \end{align*}
  $\sup_{s\in[0,t]}\norm{\psi(s,u)}\leq H_{M}$ for all $t\in [0,T]$ and the
  monotonicity of the integral, which implies that for every $s\in [0,t]$ we obtain 
  \begin{align*}
    \norm{\int_{\cHpluso}K_{\psi(s,u)}(\xi)\frac{\bP^{\perp}_{d}\E^{(t-s)B^{*}}\mu(\D\xi)}{\norm{\xi}^{2}}}
                                                                                                            &\leq
                                                                                                              H_{M}^{2}\norm{\bP^{\perp}_{d}\E^{(t-s)B^{*}}\mu(\cHpluso)}.
  \end{align*}
  Let us denote the right-hand side of~\eqref{eq:kar22-Riccati-convergence-2} by
  $K_{t}^{d}$. It follows from~\eqref{eq:kar22-Riccati-convergence-1}
  and~\eqref{eq:kar22-Riccati-convergence-2} that  
  \begin{align*}
   \norm{\psi(t,u)-\psi_{d}(t,\bP_{d}(u))} \leq
    K_{t}^{d}+L^{(1)}_{M}\int_{0}^{t}\norm{\psi(s,u)-\psi_{d}(s,\bP_{d}(u))}\,\D s, \quad
    t\in [0,T]. 
  \end{align*}
  This, the fact that $K_{t}^{d}$ is non-decreasing in $t$ and an application
  of Gronwall's inequality yields  
  \begin{align}\label{eq:kar22-Riccati-convergence-3}
    \norm{\psi(t,u)-\psi_{d}(t,\bP_{d}(u))}\leq K_{t}^{d}\exp(L^{(1)}_{M}
    t),\quad t\in [0,T],  
  \end{align}
  where we note that $\sup_{t\in [0,T]}K_{t,d}=K_{T}^{d}$. Hence taking the
  supremum over all $t\in [0,T]$ on both sides
  of~\eqref{eq:kar22-Riccati-convergence-3} yields 
  \begin{align*}
    \sup_{t\in [0,T]}\norm{\psi(t,u)-\psi_{d}(t,\bP_{d}(u))}\leq K_{T}^{d}\exp(L^{(1)}_{M} T).  
  \end{align*}
  Similarly, for the error term in~\eqref{eq:kar22-convergence-Galerkin}
  involving $\phi(\cdot,u)$, we note that $F_{d}(u)=F(u)$ for all
  $u\in\cHplus_{d}$ and hence by using~\eqref{eq:kar22-Lipschitz-Fd} we obtain
  \begin{align*}
    |\phi(t,u)-\phi_{d}(t,\bP_{d}(u))|&\leq
                                        \int_{0}^{t}|F(\psi(s,u))-F_{d}(\psi_{d}(s,\bP_{d}(u)))|\,\D
                                        s\\
                                      &\leq
                                        \big(\norm{b}+(H_{M}+1)\int_{\cHpluso}\norm{\xi}^{2}\,m(\D\xi)\big)\int_{0}^{t}K_{s,d}\E^{L^{(1)}_{M}
                                        s}\D s\\
                                      &\leq L_{M}^{(2)}
                                        t \sup_{s\in[0,t]}K_{s}^{d}\exp(L^{(1)}_{M}s), 
  \end{align*}
  with $L_{M}^{(2)}\df
  \norm{b}+(H_{M}+1)\int_{\cHpluso}\norm{\xi}^{2}\,m(\D\xi)$. Moreover, we
  conclude that the left-hand side in~\eqref{eq:kar22-convergence-Galerkin} is
  bounded by $\E^{L^{(1)}_{M}T}\big(1+L_{M}^{(2)}T\big)K^{d}_{T}$.
  We note that $L^{(1)}_{M}$ and $L_{M}^{(2)}$ are independent of $d\in\MN$ and
  thus setting $K \df \E^{L^{(1)}_{M}T}\big(1+L_{M}^{(2)}T\big)(1+TH_{M}^{2})$   
  yields~\eqref{eq:kar22-convergence-Galerkin} with $C_{T,d}$ given
  by~\eqref{eq:kar22-convergence-Galerkin-CTd}.\par{} Let us prove that $C_{T,d}$
  vanishes when $d$ tends to infinity. Indeed, note first that the map
  $t\mapsto \E^{tB^{*}}$ is continuous and thus maps compact sets to compact
  sets. In particular, for every $v\in\cHplus$ we see that the set $\set{\E^{tB^{*}}v\colon t\in[0,T]}$ is
  compact in $\cHplus$ and since for every $d\in\MN$ the operators
  $\bP^{\perp}_{d}$ converge uniformly on compact sets, we conclude that
  $\sup_{t\in[0,T]}\norm{\bP^{\perp}_{d}(\E^{tB^{*}}v)}\to 0$ as
  $d\to\infty$. Applying this to $v=u$ and $v=\mu(\cHpluso)$ accordingly, implies that the
  left-hand side in~\eqref{eq:kar22-convergence-Galerkin} converges to zero
  uniformly on compact sets in time as $d$ tends to infinity.
\end{proof}

We end this section with a corollary of Proposition~\ref{prop:kar22-existence-Riccati} providing more specific convergence rates under the additional
Assumption~\ref{assump:kar22-compact-embedding}. This convergence rate appears
again in~\eqref{eq:kar22-main-convergence-rate}. Let $V$, $H$ and $\cV$ be as in
Section~\ref{sec:exist-weak-conv}, then the following corollary holds true:

\begin{corollary}\label{cor:kar22-convergence-Galerkin-compact}
  Let the assumptions of Proposition~\ref{prop:kar22-existence-Riccati}
  hold and assume in addition that Assumption~\ref{assump:kar22-compact-embedding} is satisfied. If
  moreover $\norm{\mu(\cHpluso)}_{\cV}<\infty$ and $B^{*}(\cV_{0})\subseteq \cV_{0}$, then for all $T>0$ we have 
  \begin{align*}
    \sup_{t\in[0,T],\,\norm{u}_{\cV}\leq 1}\big(|\phi_{d}(t,\bP_{d}(u))\!-\!\phi(t,u)|\!+\!\norm{\psi_{d}(t,\bP_{d}(u))\!-\!\psi(t,u)}\big)\leq
    C_{T}\norm{\bP_{d}^{\perp}}_{\cL(\cV,\cH)},  
  \end{align*}
  with
  $C_{T}=\big(\E^{L_{C}^{(1)}}(1+L_{C}^{(2)}T)(1+TH_{C}^{2})\big)\big(\E^{T\norm{B}_{\cL(\cH)}}(1+\norm{\mu(\cHpluso)}_{\cV})\big)$
  for $H_{C}$, $L_{C}^{(1)}$ and $L_{C}^{(2)}$ being as in the proof of
  Proposition~\ref{prop:kar22-existence-Riccati} with $M=C$. 
\end{corollary}
\begin{proof}
  Let $K$ and $C_{T,d}$ be as in
  Proposition~\ref{prop:kar22-existence-Riccati}. Then we have
  $\sup_{\norm{u}_{\cV}\leq 1}K C_{T,d}\leq
  C_{T}\norm{\bP_{d}^{\perp}}_{\cL(\cV,\cH)}$. Note that since
  $B^{*}(\cV_{0})\subseteq \cV_{0}$, we have $\E^{tB^{*}}(\cV_{0})\subseteq \cV_{0}$ for all
  $t\geq 0$, which together with $\norm{\mu(\cHpluso)}_{\cV}<\infty$ implies $\norm{\bP_{d}^{\perp}\E^{tB^{*}}\mu(\cHpluso)}\leq
  \norm{\bP_{d}^{\perp}}_{\cL(\cV,\cH)}\E^{t\norm{B}_{\cL(\cH)}}\norm{\mu(\cHpluso)}_{\cV}$.   
  Similarly, for every $u\in\cHplus$ with $\norm{u}_{\cV}\leq 1$ we see that
  \begin{align*}
    \sup_{t\in [0,T]}\norm{\bP_{d}^{\perp}\E^{tB^{*}}u}\leq
    \sup_{t\in[0,T]}\norm{\bP_{d}^{\perp}\E^{tB^{*}}}_{\cL(\cV,\cH)}&\leq
                                                                      \norm{\bP_{d}^{\perp}}_{\cL(\cV,\cH)}\sup_{t\in[0,T]}\norm{\E^{tB^{*}}}_{\cL(\cH)}\\
                                                                    &\leq  \norm{\bP_{d}^{\perp}}_{\cL(\cV,\cH)}\E^{T\norm{B}_{\cL(\cH)}}. 
  \end{align*}
  Note further that $\norm{u}\leq C\norm{u}_{\cV}\leq C$ and hence from
  the proof of Proposition~\ref{prop:kar22-existence-Riccati} we see that for
  $M=C$ we can take
  $K=\E^{L_{C}^{(1)}}(1+L_{C}^{(2)}T)(1+TH_{C})$ where $H_{C}$ is such that  $\norm{\psi_{d}(t,u)|}\vee\norm{\psi(t,u)|}\leq H_{C}$ for all $t\in [0,T]$
  and $\norm{u}\leq C$,
  $L^{(1)}_{C}=\norm{B}_{\cL(\cH)}+2(H_{C}+1)\norm{\mu(\cHpluso)}$
  and $L_{C}^{(2)}=\norm{b}+2(H_{C}+1)\int_{\cHpluso}\norm{\xi}^{2}\,m(\D\xi)$.
\end{proof}


\section{Affine finite-rank operator-valued
  processes}\label{sec:kar22-affine-finite-rank} 

In this section we construct a sequence of finite-rank operator-valued
affine processes associated with the Galerkin approximations
$(\phi_{d}(\cdot,\bP_{d}(u)),(\psi_{d}(\cdot,\bP_{d}(u)))_{d\in\MN}$. The
existence of this sequence is asserted in Proposition~\ref{prop:kar22-embedding-affine-main}. First, in
Section~\ref{sec:finite-rank-valued} we project the given admissible parameter
set onto spaces of finite-rank operators and prove that for every rank
$d\in\MN$, the projected sets can be identified with \emph{matrix-valued
  admissible parameter sets} as in~\cite[Definition 2.3]{CFMT11}. In
Section~\ref{sec:ident-with-matr} we derive from this the existence of a
sequence of affine processes with values in positive semi-definite matrices
associated with a matrix-representation of the projected admissible parameters. Subsequently, in Section~\ref{sec:kar22-finite-rank-affine} we
transform this sequence back into the space of self-adjoint Hilbert-Schmidt
operators and prove that this transformed sequence satisfies the asserted properties in
Proposition~\ref{prop:kar22-embedding-affine-main}.    

\subsection{Finite-rank admissible parameters}\label{sec:finite-rank-valued}
Assume that we are in the setting of Section~\ref{sec:kar22-main-results}. In
particular, let $(b,B,m,\mu)$ be an admissible parameter set as in
Definition~\ref{def:kar22-admissible} and let $(\cH_{d},\bP_{d})_{d\in\MN}$ be
a finite-rank projectional scheme in $\cH$ (with respect to the orthonormal
basis $(\be_{i,j})_{i\leq j\in\MN}$ in $\cH$). For any two measurable
spaces $(\cE_{1},\cB_{1})$ and $(\cE_{2},\cB_{2})$ and measurable
function $f\colon (\cE_{1},\cB_{1})\to (\cE_{2},\cB_{2})$, we denote the
push-forward of a measure $\mu_{1}\colon \cE_{1}\to [0,\infty]$ with respect
to $f$ by $f_{*}\mu_{1}$, i.e. $f_{*}\mu_{1}(A)=\mu(f^{-1}(A))$ for any
$A\in\cB_{2}$ and note that $f_{*}\mu_{1}$ is a proper measure on $(\cE_{2},\cB_{2})$. For
every $d\in\MN$ we define the Borel sets
$E_{d}\df\set{\xi\in\cHplus\colon\,0<\norm{\bP_{d}(\xi)}\leq1,\,\norm{\xi}>1}$
and $E_{d}^{0}\df\cHplus\cap\set{\xi\colon\bP_{d}(\xi)\neq 0}$. Note that
$E_{d}\subseteq E_{d}^{0}\subseteq \cHpluso$ for every $d\in\MN$ and for
any (vector-valued) measure $\lambda$ on $\cB(\cHpluso)$, the
Borel-$\sigma$-algebra on $\cHpluso$, we denote the restriction of $\lambda$
to the trace of the Borel-$\sigma$-algebra generated by the open sets in
$E_{d}^{0}$ by $\lambda\vert_{E_{d}^{0}}$.  Then we introduce the following
notion:
\begin{definition}\label{def:kar22-admissible-Galerkin}
  For every $d\in\MN$ we define the parameters $(b_{d},B_{d},m_{d},\mu_{d})$
  and $M_{d}$ as follows:
  \begin{defenum}
  \item\label{item:kar22-md} The measure $m_{d}\colon \cB(\cHplus_{d}\setminus\set{0})\to
    [0,\infty]$ is defined as the push-forward of $m\vert_{E_{d}^{0}}$ with
    respect to $\bP_{d}$, i.e.
    \begin{align*}
      m_{d}(\D\xi)\df (\bP_{d\,*}m\vert_{E_{d}^{0}})(\D\xi). 
    \end{align*}
  \item The vector \label{item:kar22-bd} $b_{d}\in \cH_{d}$ is given by
    \begin{align}\label{eq:kar22-bd}
      b_{d}\df \bP_{d}(b)+\int_{\bP_{d}(E_{d})}\xi\,m_{d}(\D\xi). 
    \end{align}
  \item\label{item:kar22-mud} The $\cHplus_{d}$-valued measure $\mu_{d}\colon
    \cB(\cHplus_{d}\setminus\set{0})\to \cHplus_{d}$ is defined as the
    $\bP_{d}$-projection of the push-forward of $\mu\vert_{E_{d}^{0}}$ with respect to
    $\bP_{d}$, i.e.
    \begin{align*}
      \mu_{d}(\D\xi)\df \bP_{d}(\bP_{d\,*}\mu\vert_{E_{d}^{0}})(\D\xi).
    \end{align*}
    Moreover, we define the $\cHplus_{d}$-valued measure $M_{d}$ on
    $\cHplus_{d}\setminus\set{0}$ as follows: For every
    $A\in\cB(\cHplus_{d}\setminus\set{0})$ we set
    \begin{align}\label{eq:kar22-Md} 
      M_{d}(A)\df
      \int_{E_{d}^{0}}\one_{A}(\bP_{d}(\xi))\frac{1}{\norm{\xi}^{2}}\bP_{d}(\mu\vert_{E_{d}^{0}}(\D\xi)).
    \end{align}
  \item\label{item:kar22-Bd} The linear operator $B_{d}\colon
    \cH_{d}\to \cH_{d}$ is defined by 
    \begin{align}\label{eq:kar22-Bd}
      B_{d}(u)\df
      \bP_{d}(B(u))+\int_{\bP_{d}(E_{d})} \xi \, \langle u,
      M_{d}(\D\xi)\rangle,\quad u\in\cH_{d}. 
    \end{align}
  \end{defenum}
\end{definition}

\vspace{3mm}

\begin{remark}\label{rem:kar22-existence-integral}
  Note that by the definition of $M_{d}$ in~\eqref{eq:kar22-Md} we have
  \begin{align}\label{eq:kar22-existence-integral}
    \int_{\bP_{d}(E_{d})}\xi\,\langle
    u,M_{d}(\D\xi)\rangle=\int_{E_{d}}\frac{\bP_{d}(\xi)}{\norm{\xi}^{2}}\,\langle u, \bP_{d}(\mu\vert_{E_{d}^{0}}(\D\xi))\rangle,\quad u\in\cH_{d}, 
  \end{align}
  and for all $u\in\cH_{d}$ we have 
  \begin{align*}
    \norm{\int_{E_{d}}\frac{\bP_{d}(\xi)}{\norm{\xi}^{2}}\langle
    u,\bP_{d}(\mu\vert_{E_{d}^{0}}(\D\xi))\rangle}&\leq
    \int_{\cHplus\cap\set{\norm{\xi}>1}}\norm{\xi}^{-1}\langle u, \mu(\D\xi)\rangle\\
    &\leq \norm{u}\norm{\mu(\cHplus\cap\set{\norm{\xi}>1})}<\infty,   
  \end{align*}
  so that the integral in~\eqref{eq:kar22-Bd} is well defined (in a Bochner
  sense) and uniformly norm-bounded in $d\in\MN$. Similarly, for all $d\in\MN$ the integral part
  in~\eqref{eq:kar22-bd} satisfies
  \begin{align*}
    \int_{\bP_{d}(E_{d})}\xi\,m_{d}(\D\xi)=\int_{E_{d}}\bP_{d}(\xi)\,m(\D\xi)\leq\int_{\cHplus\cap\set{\norm{\xi}>1}}\xi\,m(\D\xi)\in\cHplus.   
  \end{align*}
\end{remark}

\vspace{3mm}

In the following we give two lemmas that we will use in the next section.

\begin{lemma}\label{lem:kar22-affine-kernel}
Let $M(x,\D\xi)\colon \mathcal{B}(\cHpluso)\to [0,\infty]$ be the kernel
defined in~\eqref{eq:kar22-affine-kernel-M} and let $\mu_{d}$ and $M_{d}$ be as
in~\cref{item:kar22-mud}. Moreover, for every $x\in\cHplus_{d}$ we define the measure
$M_{d}(x,\D\xi)\colon\cB(\cHplus_{d}\setminus\set{0})\to [0,\infty]$ by  
\begin{align}\label{eq:kar22-Mdx}
  M_{d}(x,A)\df\int_{E_{d}^{0}}\one_{A}(\bP_{d}(\xi))\frac{1}{\norm{\xi}^{2}}\langle
  x,\bP_{d}(\mu(\D\xi))\rangle,\quad A\in\cB(\cHplus_{d}\setminus\set{0}). 
\end{align}
Then for every $A\in\cB(\cHplus_{d}\setminus\set{0})$ and $x\in\cHplus_{d}$ we
have $M_{d}(x,A)=\langle x, M_{d}(A)\rangle$ and
\begin{align}\label{eq:kar22-affine-kernel-M-1}
  \langle x, \mu_{d}(A)\rangle=\int_{\cHpluso}\one_{A}(\bP_{d}(\xi))\norm{\xi}^{2}M(x,\D\xi).  
\end{align}
\end{lemma}
\begin{proof}
  Recall that $M(x,\D\xi)=\norm{\xi}^{-2}\langle x, \mu(\D\xi)\rangle$ is a
  measure on $\mathcal{B}(\cHpluso)$ and the form of this measure is
  unaffected by restricting it to the trace Borel-$\sigma$-algebras on $E_{0}^{d}$ for any
  $d\in\MN$. Now let $x\in\cHplus_{d}$ and $A\in\cB(\cHplus_{d}\setminus\set{0})$, then by
  definition we have $\langle x, \mu_{d}(A)\rangle=\langle
  x,(\bP_{d\,*}\mu\vert_{E_{d}^{0}})(A)\rangle$ and we obtain 
  \begin{align*}
    \langle
    x,(\bP_{d\,*}\mu\vert_{E_{d}^{0}})(A)\rangle&=\int_{E_{d}^{0}}\one_{\bP_{d}^{-1}(A)}(\xi)\norm{\xi}^{2}\langle
                                                  x, M\vert_{E_{d}^{0}}(\D\xi)\rangle\\
    &=\int_{\cHpluso}\one_{A}(\bP_{d}(\xi))\norm{\xi}^{2}M(x,\D\xi),
  \end{align*}
  where in the last equation we used that the integrand
  $\one_{A}(\bP_{d}(\xi))$ vanishes on the set $\set{\xi\in\cHplus\colon
    \bP_{d}(\xi)=0}$ since $\one_{A}(\bP_{d}(\xi))=\one_{A}(0)=0$ as $0\in
  A^{c}$ for $A\subseteq \cHplus_{d}\setminus\set{0}$.  
\end{proof}

In the next lemma we show that the function $F_{d}$ and $R_{d}$ from
Section~\ref{sec:galerk-appr-gener} can be expressed in terms of the
parameters $b_{d},B_{d},m_{d}$ and $M_{d}$ from
Definition~\ref{def:kar22-admissible-Galerkin}. This will help us to associate
the Galerkin approximations to affine processes in the subsequent section.

\begin{lemma}
  For every $u\in\cHplus_{d}$ we can express $F_{d}(u)$ and
  $R_{d}(u)$ by means of the parameters $b_{d},B_{d},m_{d}$ and
  $M_{d}$ as follows:
\begin{align}
  F_{d}(u)&=\langle b_{d},u\rangle-\int_{\cHplus_{d}\setminus\set{0}}\big(\E^{-\langle
                \xi,u\rangle}-1+\langle
                \chi(\xi),u\rangle\big)\,m_{d}(\D\xi),\label{eq:kar22-F-Galerkin}\\ 
  R_{d}(u)&=B^{*}_{d}(u)-\int_{\cHplus_{d}\setminus\set{0}}\big(\E^{-\langle \xi,u\rangle}-1+\langle
                \chi(\xi),u\rangle\big)\,M_{d}(\D\xi).\label{eq:kar22-R-Galerkin}
\end{align} 
\end{lemma}
\begin{proof}
  We only proof the identity~\eqref{eq:kar22-R-Galerkin} as the proof
  of~\eqref{eq:kar22-F-Galerkin} is similar. Let $u\in\cHplus_{d}$ and note
  that we have $\langle \xi, u\rangle=\langle \bP_{d}(\xi), u\rangle$ for
  every $\xi\in\cHplus$ since $\bP_{d}$ is an orthogonal projection with
  respect to $\langle\cdot,\cdot\rangle$. Setting $M(\D\xi)\df
  \norm{\xi}^{-2}\mu(\D\xi)$ we see from the definition of $R_{d}$,
  \cref{item:kar22-Bd} and~\eqref{eq:kar22-existence-integral} that  
  \begin{align}
    R_{d}(u)&=\bP_{d}(B(u))-\int_{\cHpluso}\big(\E^{-\langle
              \bP_{d}(\xi), u\rangle}-1+\langle
              \bP_{d}(\chi(\xi)),u\rangle\big)\frac{\bP_{d}(\mu(\D\xi))}{\norm{\xi}^{2}}\nonumber\\
            &=B_{d}(u)-\int_{\cHpluso}\left(\E^{-\langle\bP_{d}(\xi),u\rangle}-1+\langle
              \chi(\bP_{d}(\xi)),u\rangle\right)\,\bP_{d}(M(\D\xi)).\label{eq:kar22-truncation-change}
  \end{align}
  Note that on the set $E_{d}^{0}=\cHplus\cap\set{\xi\in\cHplus \colon\,\bP_{d}(\xi)=0}$ the
  integrand on the right-hand side of~\eqref{eq:kar22-truncation-change}
  vanishes and hence we see that the integral coincides with
  \begin{align}\label{eq:kar22-truncation-change-2}
    \int_{E_{d}^{0}}\left(\E^{-\langle\bP_{d}(\xi),u\rangle}-1+\langle\chi(\bP_{d}(\xi)),u\rangle\right)\,\bP_{d}(M\vert_{E_{d}^{0}}(\D\xi)).  
  \end{align}
  From the definition of $M_{d}$ in~\eqref{eq:kar22-Md} and by the
  change-of-variables formula for push-forward measures, we conclude that the
  integral in~\eqref{eq:kar22-truncation-change-2} is equal to
  \begin{align*}
    \int_{\cHplus_{d}\setminus\set{0}}\big(\E^{-\langle\xi,u\rangle}-1+\langle
    \chi(\xi),u\rangle\big)\,M_{d}(\D\xi),
  \end{align*}
  which inserted back into~\eqref{eq:kar22-truncation-change} proves the identity~\eqref{eq:kar22-R-Galerkin}.
\end{proof}

\subsection{Identification with matrix-valued affine processes}\label{sec:ident-with-matr}
For every $d\in\MN$ we denote by
$(\mathbb{M}_{d},\langle\cdot,\cdot\rangle_{d})$ the space of
all real $d\times d$-matrices equipped with the trace inner-product $\langle
x,y\rangle_{d}\df \Tr(y^{\intercal}x)$ for $x,y\in\mathbb{M}_{d}$, where $y^{\intercal}\in\mathbb{M}_{d}$
denotes the transpose of $y$. The norm $\norm{\cdot}_{d}$ induced by $\langle
\cdot,\cdot\rangle_{d}$ is called the \textit{Frobenius norm}, which is
nothing else than the Hilbert-Schmidt norm in the case of $H=\MR_{d}$. Let us
denote the subspace of $\mathbb{M}_{d}$ consisting of all the symmetric $d\times
d$-matrices by $\MS_{d}$. For $d\in\MN$ we denote by
$\set{v_{1},\ldots,v_{d}}$ the standard basis of $\MR_{d}$ and define the
coordinate system $\Phi_{d}\colon \MR_{d}\to H_{d}$ associated with the basis
$\set{e_{1},\ldots, e_{d}}$ of $H_{d}$ by   
\begin{align}\label{eq:kar22-coordinate-system}
  \Phi_{d}(v_{i})=e_{i},\quad \text{ for } i=1,\ldots,d.  
\end{align}
The coordinate system $\Phi_{d}$ identifies the $d$-dimensional subspace $H_{d}$
with $\MR_{d}$ and we can represent every linear operator $A\in\cL(H_{d})$ as
a $d\times d$-matrix by using the mapping $i_{d}\colon\cL(H_{d})\to\MS_{d}$
given by  
\begin{align}\label{eq:kar22-canonical-identification-d}
  i_{d}(A)\df\Phi^{-1}_{d}\circ A\circ \Phi_{d},  
\end{align}
where under the usual matrix-identification we shall understand $i_{d}(A)$ as an
element in $\mathbb{M}_{d}$. Note that whenever $A$ is
self-adjoint, its matrix representation $i_{d}(A)$ is self-adjoint as
well, which can be seen by taking $x,y\in\MR_{d}$ and the brief computation
\begin{align*}
  (i_{d}(A)x,y)_{\MR^{d}}=(A\circ
  \Phi_{d}(x),\Phi_{d}(y))_{H}=(\Phi_{d}(x),A^{*}\circ\Phi_{d}(y))_{H}=(x,i_{d}(A)y)_{\MR^{d}}.  
\end{align*}
Under the mapping $i_{d}$ in~\eqref{eq:kar22-canonical-identification-d} we
identify $\cH_{d}\vert_{H_{d}}\subseteq \cL(H_{d})$ with $\MS_{d}$ and note that
$i_{d}$ is an isometry between $\MS_{d}$ and $\cH_{d}\vert_{H_{d}}$, i.e. it
identifies the Frobenius with the Hilbert-Schmidt norm. In the following, we
sometimes omit writing our the restriction $\rvert_{H_{d}}$, when it is clear
from the context. Moreover, we denote by $\MS_{d}^{+}$ the convex
cone of all symmetric positive semi-definite $d\times d$-matrices and observe
that the positivity is preserved under $i_{d}$,
i.e. $i_{d}(\cHplus_{d})=\MS_{d}^{+}$. In the following definition we
introduce yet another transformation of the parameters $b_{d}$, $B_{d}$ ,
$m_{d}$, $\mu_{d}$ and $M_{d}$ from
Definition~\ref{def:kar22-admissible-Galerkin}. This time by identifying the
Hilbert spaces $\cH_{d}$ and $\MS_{d}$:   

\begin{definition}\label{def:kar22-admissible-Galerkin-matrix}
  Let $(b,B,m,\mu)$ be an admissible parameter set as in
  Definition~\ref{def:kar22-admissible} and for $d\in\MN$ let
  $(b_{d},B_{d},m_{d},\mu_{d})$ and $M_{d}$ be as in
  Definition~\ref{def:kar22-admissible-Galerkin}. For every $d\in\MN$ we define the parameters 
  $(\tilde{b}_{d},\tilde{B}_{d},\tilde{m}_{d},\tilde{\mu}_{d})$ and $\tilde{M}_{d}$ as follows:
  \begin{defenum}
  \item\label{item:kar22-b-d} The matrix $\tilde{b}_{d}\in\MS_{d}^{+}$ is
    defined as $\tilde{b}_{d}\df i_{d}(b_{d})$.
  \item\label{item:kar22-B-d} The linear operator $\tilde{B}_{d}\colon
    \MS_{d}\to\MS_{d}$ is given by $\tilde{B}_{d}\df i_{d}\circ B_{d}\circ i_{d}^{-1}$. 
  \item\label{item:kar22-m-d} The measure $\tilde{m}_{d}\colon\cB(\MS_{d}^{+}\setminus\set{0})\to
    [0,\infty]$ is defined as the push-forward of $m_{d}$ with respect to $i_{d}$, i.e.
    \begin{align*}
     \tilde{m}_{d}(\D\xi)\df (i_{d\,*}m_{d})(\D\xi). 
    \end{align*}
  \item\label{item:kar22-mu-d} The matrix-valued measure $\tilde{\mu}_{d}\colon
    \cB(\MS_{d}^{+}\setminus\set{0})\to \MS_{d}^{+}$ is defined as the
    composition of $i_{d}$ and the push-forward of $\mu_{d}$ with respect to $i_{d}$, i.e.
    \begin{align*}
      \tilde{\mu}_{d}(\D\xi)=i_{d}((i_{d\,*}\mu_{d})(\D\xi)). 
    \end{align*}
    Moreover, we define the $\MS_{d}^{+}$-valued measure $\tilde{M}_{d}(\D\xi)$ as follows: For every
    $A\in\cB(\MS_{d}^{+}\setminus\set{0})$ we set 
  \begin{align*} 
    \tilde{M}_{d}(A)=\int_{\cHpluso}\one_{A}(i_{d}(\bP_{d}(\xi)))\frac{1}{\norm{\xi}^{2}}i_{d}(\bP_{d}(\mu(\D\xi))),
  \end{align*}
  and for every $x\in\MS_{d}^{+}$ we write $\tilde{M}_{d}(x,\D\xi)\df\langle x,\tilde{M}_{d}(\D\xi)\rangle$.
\end{defenum}
\end{definition}

Now, let $\chi_{d}\colon \MS_{d}\to\MS_{d}$ be defined as
$\chi_{d}(\xi)=\xi\one_{\norm{\xi}_{d}\leq 1}(\xi)$. In the next lemma we show
some crucial properties of the parameters
$(\tilde{b}_{d},\tilde{B}_{d},\tilde{m}_{d},\tilde{\mu}_{d})$ and
$\tilde{M}_{d}$.  
\begin{lemma}\label{lem:kar22-admissible-matrix-sense}
  Let $d\in\MN$ and $\tilde{b}_{d}$, $\tilde{B}_d$, $\tilde{m}_{d}$
  and $\tilde{M}_{d}$ defined as in
  Definition~\ref{def:kar22-admissible-Galerkin-matrix}. Then the following holds true: 
  \begin{lemenum}
  \item\label{item:kar22-m-matrix} $\int_{\MS_{d}^{+}\setminus\set{0}}
    \big(\norm{\xi}_{d}\vee
    \norm{\xi}^{2}_{d}\big)\tilde{m}_{d}(\D\xi)<\infty$.
  \item\label{item:kar22-b-matrix}
    $\tilde{b}_{d}-\int_{\MS_{d}^{+}\setminus\set{0}}\chi_{d}(\xi)\tilde{m}_{d}(\D\xi)\in\MS_{d}^{+}$.
  \item\label{item:kar22-mu-matrix} For every $A\in\cB(\MS_{d}^{+}\setminus\set{0})$ we
    have $\tilde{M}_{d}(A)\in\MS_{d}^{+}$ and
    \begin{align*}
      \int_{\MS_{d}^{+}\setminus\set{0}}\langle \chi_{d}(\xi),u\rangle_{d} \tilde{M}_{d}(x,\D\xi)<\infty,
    \end{align*}
    for all $x,u\in\MS_{d}^{+}$ such that $\langle x,u\rangle_{d}=0$.
   \item\label{item:kar22-second-moment-mu-d} For all $x\in\MS_{d}^{+}$ we
     have $\int_{\MS_{d}^{+}\setminus\set{0}} \norm{\xi}^2\tilde{M}_{d}(x,\D\xi)<\infty$.
  \item\label{item:kar22-B-matrix}  We have
    \begin{align}\label{eq:kar22-B-matrix-quasi-monotone}
      \langle \tilde{B}_d(x),u\rangle_{d}-\int_{\cHplus}\langle
      \chi_{d}(\xi),u\rangle_{d}\,\langle x,\tilde{M}_{d}(\D\xi)\rangle_{d}\geq 0, 
    \end{align}
    for all $x,u\in\MS_{d}^{+}$ such that $\langle x,u\rangle_{d}=0$.
  \end{lemenum}
\end{lemma}
\begin{proof}
  First, note that for every $E\in\cB(\MS^{+}_{d}\setminus\set{0})$ we have 
  \begin{align*}
    i_{d\,*}(\bP_{d\,*}m)(E)=m(\bP_{d}^{-1}(i_{d}^{-1}(E)))=m((i_{d}\circ\bP_{d})^{-1}(E))=(i_{d}\circ\bP_{d})_{\,*}m(E), 
  \end{align*}
  and $\tilde{m}_{d}(\D\xi)=((i_{d}\circ\bP_{d})_{\,*}m\rvert_{E_{d}^{0}})(\D\xi)$ by definition
  and the analogous statement holds for the measure $\tilde{M}_{d}$. To show
  \cref{item:kar22-m-matrix} we split the integral into two parts
  \begin{align}
    \int_{\MS_{d}^{+}\setminus\set{0}} \big(\norm{\xi}_{d}\vee
    \norm{\xi}^{2}_{d}\big)\,\tilde{m}_{d}(\D\xi)&=\int_{\set{\xi\in\MS_{d}^{+}\colon
    0<\norm{\xi}_{d}\leq 1}}
    \norm{\xi}_{d}\,\tilde{m}_{d}(\D\xi)\nonumber \\
    &\quad+\int_{\set{\xi\in\MS_{d}^{+}\colon\norm{\xi}_{d}>1}}\norm{\xi}^{2}_{d}\,\tilde{m}_{d}(\D\xi),\label{eq:kar22-m-matrix-2}
  \end{align}
  and consider the two integrals on the right-hand side
  of~\eqref{eq:kar22-m-matrix-2} separately. By the change-of-variable formula
  for pushforward measures and since $i_{d}$ is an isometry,
  i.e. $\norm{\xi}=\norm{i_{d}(\xi)}_{d}$ for $\xi\in\cHplus_{d}$, we deduce the following for the first integral in~\eqref{eq:kar22-m-matrix-2} 
  \begin{align}
    \int_{\set{\xi\in\MS_d^{+}\colon 0<\norm{\xi}_{d}\leq 1}} \| \xi \|_{d}
    \,\tilde{m}_{d}(\D\xi)&=\int_{\set{\xi\in\cH_{d}^{+}\colon
                            0<\norm{i_{d}(\xi)}_{d}\leq 1}}\norm{i_{d}(\xi)}_{d}\,m_{d}(\D\xi) \nonumber\\  
                          &=\int_{\set{\xi\in\cH_{d}^{+}\colon
                                 0<\norm{\xi}\leq 1}}\norm{\xi}\,m_{d}(\D\xi)\nonumber\\
                          &=\int_{\set{\xi\in\cH^{+}\colon
                                 0<\norm{\bP_{d}(\xi)}\leq 1}} \| \bP_{d}(\xi) \|
                                 \,m(\D\xi)\\
                          &\leq \int_{\cHpluso}
                                 (\sum_{i=1}^{d}\sum_{j=i}^{d}\langle \xi, \be_{i,j}\rangle^{2})^{\frac{1}{2}}
                                 \,m(\D\xi)\nonumber\\
                          &\leq  \sum_{i=1}^{d}\sum_{j=i}^{d} \int_{\cHpluso} |\langle \xi, \be_{i,j}\rangle|
                                 \,m(\D\xi) <\infty \label{eq:kar22-const-jump-check-2}
  \end{align}  
  where the inequality in~\eqref{eq:kar22-const-jump-check-2} follows from
  part (b) in \cref{item:kar22-m-2moment}, which yields
  \begin{align*}
    \int_{\cH^+ \cap \{ 0 < \| \xi \| \leq 1 \}}|\langle
    \xi,\be_{i,j}\rangle|\,m(\D\xi)<\infty,\quad\forall\, 1\leq i\leq j\leq d,
  \end{align*}
  together with part (a) of \cref{item:kar22-m-matrix} which yields
  \begin{align*}
   \int_{\cH^+ \cap \{\| \xi \| >1 \}}|\langle
    \xi,\be_{i,j}\rangle|\,\dm\leq \int_{\cH^+ \cap \{\| \xi \| >1
    \}}\norm{\xi}^{2}\,\dm<\infty,
  \end{align*}
  for all $1\leq i\leq j\leq d$. Similarly , for the second integral on the right-hand
  side of~\eqref{eq:kar22-m-matrix-2} we see that 
  \begin{align*}
    \int_{\set{\xi\in\MS_{d}^{+}\colon \norm{\xi}_{d}>1}}\norm{\xi}_{d}^2\,\tilde{m}_{d}(\D\xi)&
    \leq \int_{\cHplus\setminus\set{0}} \norm{\bP_{d}(\xi)}^2\,m(\D\xi)\\
                                                                             &\leq\int_{\cHplus\setminus\set{0}} \norm{\xi}^2m(\D\xi)<\infty,  
  \end{align*}
   which follows again from part (b) in \cref{item:kar22-m-2moment}.
   Next we show \cref{item:kar22-b-matrix}. By definition we have
   \begin{align*}
     \tilde{b}_{d}=i_{d}(b_{d})=i_{d}(\bP_{d}(b))+\int_{i_{d}(\bP_{d}(E_{d}))}\xi\,\tilde{m}_{d}(\D\xi) 
   \end{align*}
   and from \cref{item:kar22-drift} it follows that $b\in\cHplus$ and
   \begin{align*}
     \int_{i_{d}(\bP_{d}(E_{d}))}\xi\,\tilde{m}_{d}(\D\xi)=\int_{E_{d}}i_{d}(\bP_{d}(\xi))\,m(\D\xi).  
   \end{align*}
   Now, since $\bP_{d}(\cHplus)=\cHplus_{d}$ and also $i_d(\cHplus_{d})=\MS_{d}^{+}$ we see that
   $\tilde{b}_{d}\in\MS_{d}^{+}$. Moreover, since
   $\one_{E_{d}}-\one_{\cHplus \cap \set{0<\norm{\bP_{d}(\xi)}\leq
       1}}=-\one_{\cHplus\cap\set{0<\norm{\xi}\leq 1}}$ and
   \begin{align*}
   \bP_{d}(I_{m})&=\sum_{i=1}^{d}\sum_{j=i}^{d}\langle
     I_{m},\be_{i,j}\rangle\be_{i,j}=\sum_{i=1}^{d}\sum_{j=i}^{d}\big(\int_{\cHpluso}\langle
     \chi(\xi),\be_{i,j}\rangle\,m(\D\xi)\big)\be_{i,j}\\
     &=\int_{\cHpluso}\bP_{d}(\chi(\xi))\,m(\D\xi),  
   \end{align*}
   we conclude that
   \begin{align*}
     \tilde{b}_{d}-\int_{\MS_{d}^{+}\setminus\set{0}}\chi_{d}(\xi)\tilde{m}_{d}(\D\xi)=i_{d}\big(\bP_{d}\big(b-I_{m}\big)\big)\geq_{\MS_{d}^{+}} 0, 
   \end{align*}
   where it follows from~\cref{item:kar22-drift} that $b-I_{m}\in\cHplus$.
   We continue with \cref{item:kar22-mu-matrix} and show first that the measure
   $\tilde{M}_{d}$ is a sigma-finite measure on $\MS_{d}^{+}\setminus\set{0}$ 
   such that for every $A\in\cB(\MS_{d}^{+}\setminus\set{0})$ we have 
   $\tilde{M}_{d}(A)\in\MS_{d}^{+}$. For this, note that by definition
   $\mu(E)\in\cHplus$ for all $E\in\cB(\cHpluso)$ and the same holds of course
   for its restriction to the set $E_{d}^{0}$. Hence, this applied to the measurable set $(i_{d}\circ\bP_{d})^{-1}(A)\subseteq
   \cHplus_{d}$ gives $\tilde{M}_{d}(A)\in\MS_{d}^{+}$. Note that the
   kernel $\tilde{M}_{d}(x,\D\xi)\colon\cB(\MS_{d}^{+}\setminus\set{0})\to
   [0,\infty]$, for all $A\in\cB(\MS_{d}^{+}\setminus\set{0})$, satisfies 
  \begin{align*}
    \tilde{M}_{d}(x,A)=\int_{\cHpluso}\one_{A}(i_{d}(\bP_{d}(\xi)))\frac{1}{\norm{\xi}^{2}}\langle
    x, i_{d}(\bP_{d}(\mu(\D\xi)))\rangle_{d},  
  \end{align*}
  or equivalently $\tilde{M}_{d}(x,\D\xi)=i_{d}(i_{d\,*}M_{d}(x,\D\xi))$.   
  Moreover, let $x,u\in\MS_{d}^{+}$ such that $\langle x,u\rangle_{d}=0$, then
    \begin{align}
      \int_{\MS_{d}^{+}\setminus\set{0}}\langle \chi_{d}(\xi), u\rangle_{d}
      \tilde{M}_{d}(x,\D\xi) &=\int_{\cHpluso}\langle
                              \chi_{d}(i_{d}(\bP_{d}(\xi))),u\rangle_{d} \frac{\langle
                              x,
                               i_{d}(\bP_{d}(\mu(\D\xi)))\rangle_{d}}{\norm{\xi}^{2}}\nonumber\\
                             &= \int_{\cHpluso}\langle
                              \chi(\bP_{d}(\xi)),i_{d}^{-1}(u)\rangle \frac{\langle
                               i_{d}^{-1}(x),\bP_{d}(\mu(\D\xi))\rangle}{\norm{\xi}^{2}}\nonumber\\
                             &\leq \int_{\cHpluso} \langle
                               \bP_{d}(\xi), i_{d}^{-1}(u)\rangle \frac{\langle
                               i_{d}^{-1}(x),\mu(\D\xi)\rangle}{\norm{\xi}^{2}}\nonumber\\
      &=\int_{\cHpluso} \langle \xi, i_{d}^{-1}(u)\rangle \frac{\langle
                               i_{d}^{-1}(x),\mu(\D\xi)\rangle}{\norm{\xi}^{2}}<\infty,\label{eq:kar22-mu-matrix-1}
    \end{align}
    where the last inequality~\eqref{eq:kar22-mu-matrix-1} follows from
    \cref{item:kar22-affine-kernel} and since $i_{d}^{-1}(x)$,
    $i_{d}^{-1}(u)\in\cHplus$ satisfy $\langle
    i_{d}^{-1}(x),i_{d}^{-1}(u)\rangle=\langle x, u\rangle_{d}=0$. The
    property in~\eqref{eq:kar22-second-moment-mu-d} follows from 
    \begin{align*}
       \int_{\MS_{d}^{+}\setminus\set{0}}\norm{\xi}_{d}^{2}\tilde{M}_{d}(x,\D\xi)&\leq
                                                   \int_{\cHpluso}\norm{\bP_{d}(\xi)}^{2}\frac{\langle
                                                                                   i_{d}^{-1}(x),
                                                                                   \mu(\D\xi)\rangle}{\norm{\xi}^{2}}\\
      &\leq \big\langle i_{d}^{-1}(x),\mu(\cHpluso)\big\rangle<\infty. 
    \end{align*}
    Finally, we show \cref{item:kar22-B-matrix}. For this let
    $x,u\in\MS_{d}^{+}$ be such that $\langle x,u\rangle_{d}=0$ and note that
    \begin{align*}
      \int_{\MS_{d}^{+}\cap\set{0<\norm{\xi}_{d}\leq 1}}\langle
      \chi_{d}(\xi),u\rangle_{d}\,\tilde{M}_{d}(x,\D\xi)=
      \!\int_{\cHplus\cap \set{0<\norm{\bP_{d}(\xi)}\leq 1}}\!\langle
      \xi ,i_{d}^{-1}(u)\rangle\,M(i_{d}^{-1}(x),\D\xi),  
    \end{align*}
    as well as
    \begin{align*}
      \langle \tilde{B}_d(x), u\rangle_{d}=
      \langle B(i_{d}^{-1}(x)),i_{d}^{-1}(u)\rangle_{d}+\int_{E_{d}}\langle
      \bP_{d}(\xi), i_{d}^{-1}(u)\rangle_{d}\,M(i_{d}^{-1}(x),\D\xi).
    \end{align*}
    Now again, as $\one_{E_{d}}-\one_{\cHplus \cap \set{0<\norm{\bP_{d}(\xi)}\leq
        1}}=-\one_{\cHplus\cap\set{0<\norm{\xi}\leq 1}}$ and 
    $\langle i_{d}^{-1}(x), i_{d}^{-1}(u)\rangle=0$ we conclude the
    inequality~\eqref{eq:kar22-B-matrix-quasi-monotone} from
    \begin{align*}
      \langle B(i_{d}^{-1}(x)),i_{d}^{-1}(u)\rangle-
      \int_{\cHplus\cap \set{0<\norm{\xi}\leq 1}}\langle
      \xi,i_{d}^{-1}(u)\rangle\,M(i_{d}^{-1}(x),\D\xi)\geq 0,
    \end{align*}
    which holds true by \cref{item:kar22-linear-operator} and proves the
    last assertion of Lemma~\ref{lem:kar22-admissible-matrix-sense}.
  \end{proof}

  \vspace{1mm}

  For $d\in\MN$ let us denote by $D(\MRplus,\MS_{d}^{+})$ the Skorohod space
  of all c\`adl\`ag path from $\MRplus$ into $\MS_{d}^{+}$ and let
  $\cB(D(\MRplus,\MS_{d}^{+}))$ be the Borel-$\sigma$-algebra on
  $D(\MRplus,\MS_{d}^{+})$ with respect to the Skorohod topology. In the
  following proposition we assert the existence of a unique affine process on
  $\MS_{d}^{+}$ associated with an \emph{matrix-valued admissible parameter
    set} built from the parameters  $\tilde{b}_{d}$, $\tilde{B}_d$,
  $\tilde{m}_{d}$ and $\tilde{M}_{d}$ and with paths in $D(\MRplus,\MS_{d}^{+})$.

  \vspace{1mm}

  \begin{proposition}\label{prop:kar22-Galerkin-explicit}
    Let $d\in\MN$ and
    $(\tilde{b}_{d},\tilde{B}_d,\tilde{m}_{d},\tilde{\mu}_{d})$ and
    $\tilde{M}_{d}$ be as in
    Definition~\ref{def:kar22-admissible-Galerkin-matrix}. Then there exists a 
    unique Markov process
    $(\tilde{X}^{d},(\tilde{\MP}_{x})_{x\in\MS_{d}^{+}})$, with paths in
    $D(\MRplus,\MS_{d}^{+})$ and where $\tilde{\MP}_{x}^{d}$ denotes the law
    of $\tilde{X}^{d}$ given $\tilde{X}_{0}^{d}=x\in\MS_{d}^{+}$, such that
    for every $x\in\MS_{d}^{+}$ we have    
    \begin{align}\label{eq:kar22-affine-property-matrix}
      \EXspec{\tilde{P}_{x}^{d}}{\E^{-\langle
      \tilde{X}^{d}_{t},u\rangle_{d}}}=\E^{-\tilde{\phi}_{d}(t,u)-\langle
      x,\tilde{\psi}_{d}(t,u)\rangle_{d}},\quad t\geq 0,\,u\in\MS_{d}^{+},  
    \end{align}
    for $(\tilde{\phi}_{d}(\cdot,u),\tilde{\psi}_{d}(\cdot,u))$ the
    unique solution of the following equations: 
    \begin{subequations}
      \begin{empheq}[left=\empheqlbrace]{align}
        \,\frac{\partial \tilde{\phi}_{d}(t,u)}{\partial
          t}&=\tilde{F}_{d}(\tilde{\phi}_{d}(t,u)),
        &\tilde{\phi}_{d}(0,u)=0,\label{eq:kar22-phi-matrix}\\
        \,\frac{\partial \tilde{\psi}_{d}(t,u)}{\partial t}&=\tilde{R}_{d}(\tilde{\psi}_{d}(t,u)),
        &\tilde{\psi}_{d}(0,u)=u,\label{eq:kar22-psi-matrix}
      \end{empheq}
    \end{subequations}
    where the functions $\tilde{F}_{d}\colon \MS_{d}^{+}\to \MR$ and $\tilde{R}_{d}\colon
    \MS_{d}^{+}\to \MS_{d}$ are given by
    \begin{align*}
      \tilde{F}_{d}(u)&\df\langle \tilde{b}_{d},u\rangle_{d}-\int_{\MS_{d}^{+}\setminus\set{0}}\big(\E^{-\langle
                        \xi,u\rangle_{d}}-1+\langle\chi_{d}(\xi),u
                        \rangle_{d}\big)\,\tilde{m}_{d}(\D\xi),\\      
      \tilde{R}_{d}(u)&\df\tilde{B}^{*}_d(u)-\int_{\MS_{d}^{+}\setminus\set{0}}\big(\E^{-\langle
                        \xi,u\rangle_{d}}-1+\langle
                        \chi_{d}(\xi),u \rangle_{d}\big)\,\tilde{M}_{d}(\D\xi).
    \end{align*}
    Moreover, for all $x\in\MS_{d}^{+}$ the process $\tilde{X}^{d}$ satisfies $\tilde{\MP}_{x}(\set{\tilde{X}_{t}^{d}\in\MS^{+}_{d}\colon\, t\geq
      0})=1$ and is a square-integrable semimartingale on $\MS_{d}^{+}$ whose
    semimartingale characteristics 
    $(\tilde{A}^{d},\tilde{C}^{d},\tilde{\nu}^{d})$, with respect to
    $\chi_{d}$, are given by
    \begin{align*}
      \tilde{A}_{t}^{d}&=\int_{0}^{t}\big(\tilde{b}_{d}+\tilde{B}_d(\tilde{X}_{s}^{d})\big)\,\D
                         s, \quad \tilde{C}_{t}^{d}=0, \quad \tilde{\nu}^{d}(\D t,\D\xi)=
                         \big(\tilde{m}_{d}(\D\xi)+\tilde{M}_{d}(\tilde{X}_{t}^{d},\D\xi)\big)\D t. 
    \end{align*}
  \end{proposition}
  \begin{proof}
    Given the parameters $\tilde{b}_{d}$, $\tilde{B}_d$, $\tilde{m}_{d}$ and
    $\tilde{M}_{d}$ we define the following adjusted constant and linear drift
    parameters $\tilde{c}_{d}$ and $\tilde{D}_{d}(u)$ for $u\in\MS_{d}$ as
    \begin{align*}
      \tilde{c}_{d}\df
      \tilde{b}_{d}-\int_{\MS_{d}^{+}\setminus\set{0}}\chi_{d}(\xi)\tilde{m}_{d}(\D\xi),\quad \tilde{D}_{d}(u)\df \tilde{B}^{*}_d(u)-\int_{\MS_{d}^{+}\setminus\set{0}}\langle
      \chi_{d}(\xi),u\rangle_{d}\,\tilde{M}_{d}(\D\xi).  
    \end{align*}
    It then follows from the properties in~\cref{item:kar22-b-matrix}
    and~\eqref{eq:kar22-B-matrix-quasi-monotone} that
    $\tilde{c}_{d}\in\MS_{d}^{+}$ and $\langle
    \tilde{D}_{d}(u),x\rangle_{d}\geq 0$ for all $u,x\in\MS_{d}^{+}$ with $\langle u,x\rangle_{d}=0$. Together with the other properties shown
    in Lemma~\ref{lem:kar22-admissible-matrix-sense} we conclude that the
    parameter set
    $(0,\tilde{c}_{d},\tilde{D}_{d},0,0,\tilde{m}_{d},\tilde{M}_{d})$ is
    an \emph{admissible parameter set}  for $\MS_{d}^{+}$-valued affine
    process according to~\cite[Definition 3.1]{May12}. It thus follows from
    \cite[Theorem 2.4]{CFMT11} and~\cite[Theorem 3.2]{May12} that there exists
    a unique affine process $(\tilde{X}^{d})_{t\geq 0}$ with values in
    $\MS_{d}^{+}$ such that for every $t\geq 0$ and $u\in\MS_{d}^{+}$ the
    affine transform formula~\eqref{eq:kar22-affine-property-matrix} holds with
    $(\tilde{\phi}_{d}(\cdot,u),\tilde{\psi}_{d}(\cdot,u))$ being the unique
    solution to the following equations:  
    \begin{subequations}
      \begin{empheq}[left=\empheqlbrace]{align}
        \frac{\partial \tilde{\phi}_{d}(t,u)}{\partial t}=\langle
        \tilde{c}_{d},\tilde{\psi}_{d}(t,u)\rangle_{d}-\int_{\MS_{d}^{+}\setminus\set{0}}\big(\E^{-\langle
          \xi,\tilde{\psi}_{d}(t,u)\rangle_{d}}-1\big)\,\tilde{m}_{d}(\D\xi),\label{eq:kar22-phi-matrix-1}\\
        \frac{\partial \tilde{\psi}_{d}(t,u)}{\partial t}=
        \tilde{D}_{d}(\tilde{\psi}_{d}(t,u))-\int_{\MS_{d}^{+}\setminus\set{0}}\big(\E^{-\langle
          \xi,\tilde{\psi}_{d}(t,u)\rangle_{d}}-1\big)\,\tilde{M}_{d}(\D\xi),\label{eq:kar22-psi-matrix-1}
      \end{empheq}
    \end{subequations}
    and initial conditions $\tilde{\psi}_{d}(0,u)=u$ and
    $\tilde{\phi}_{d}(0,u)=0$. Inserting $\tilde{c}_{d}$ and $\tilde{D}_{d}$
    into~\eqref{eq:kar22-phi-matrix-1}-\eqref{eq:kar22-psi-matrix-1} proves
    the equivalence with
    equations~~\eqref{eq:kar22-phi-matrix}-\eqref{eq:kar22-psi-matrix}.\par{}
    The existence of a c\`adl\`ag version follows from~\cite{CT13} and we
    shall denote this version again by $(\tilde{X}^{d})_{t\geq 0}$. Moreover,
    we denote the law of $\tilde{X}^{d}$ given that $\tilde{X}^{d}_{0}=x$ by
    $\tilde{\MP}^{d}_{x}$. Note that the first, fourth and fifth
    component of
    $(0,\tilde{c}_{d},\tilde{D}_{d},0,0,\tilde{m}_{d},\tilde{M}_{d})$ are
    zero, which correspond to a vanishing \emph{diffusion} component as well
    as the absence of constant and linear \emph{killing terms}. By~\cite[Remark
    2.5]{CFMT11}, this together with the moment assumption
    in~\cref{item:kar22-m-matrix} and~\eqref{eq:kar22-second-moment-mu-d}
    implies that the $\MS_{d}^{+}$-valued affine process
    $(\tilde{X}^{d})_{t\geq 0}$ satisfies
    $\tilde{\MP}_{x}(\set{\tilde{X}_{t}^{d}\in\MS^{+}_{d}\colon\, t\geq
      0})=1$. For every $d\in\MN$ and $x\in\MS_{d}^{+}$ the law
    $\tilde{\MP}^{d}_{x}$ is thus defined on
    $\cB(D(\MRplus,\MS_{d}^{+}))$. Moreover, as the diffusion part is zero, it
    follows from~\cite{May12} that $(\tilde{X}^{d})_{t\geq 0}$ is of finite-variation.\par{}
    
    Moreover, it follows from \cite[Theorem 2.6]{CFMT11} that the process
    $\tilde{X}_{d}$ is a semimartingale with characteristics given by
    $\tilde{C}^{d}_{t}=0$, $\tilde{\nu}^{d}(\D t,\D\xi)=
    \big(\tilde{m}_{d}(\D\xi)+\tilde{M}_{d}(\tilde{X}_{t}^{d},\D\xi)\big)\D t$ 
    and 
    \begin{align*}
      \tilde{A}^{d}&=\int_{0}^{t}\big(\tilde{c}_{d}+\int_{\MS_{d}^{+}\setminus\set{0}}\chi_{d}(\xi)\,\tilde{m}_{d}(\D\xi)+\tilde{D}_{d}(\tilde{X}^{d}_{s})+\int_{\MS_{d}^{+}\setminus\set{0}}\chi_{d}(\xi)\langle
                     \tilde{X}^{d}_{s},\tilde{M}_{d}(\D\xi)\rangle\big)\,\D
                     s\\
                   &=\int_{0}^{t}\big(\tilde{b}_{d}+\tilde{B}_d(\tilde{X}^{d}_{s})\big)\,\D s,
    \end{align*}
    which proves the asserted form of the characteristic triplet
    $(\tilde{A}^{d},\tilde{C}^{d},\tilde{\nu}^{d})$. Lastly, we note that the
    process $(\tilde{X}_{t}^{d})_{t\geq 0}$ is of finite-variation, hence locally
    bounded and by Lemma~\ref{lem:kar22-admissible-matrix-sense} we conclude
    that $\int_{0}^{t}\int_{\MS_{d}^{+}\setminus\set{0}}\norm{\xi}_{d}^{2}\,\tilde{\nu}^{d}(\D
    t,\D\xi)<\infty$ for all $t\geq 0$, which by~\cite[Proposition 2.29
    b)]{JS03} implies that $X$ is a square-integrable martingale,
    i.e. $\EXspec{\tilde{P}_{x}^{d}}{\norm{\tilde{X}^{d}_{t}}_{d}^{2}}<\infty$
    for all $t\geq 0$.  
\end{proof}

As a corollary from \cite{May12} we can sharpen the property in
\cref{item:kar22-mu-matrix}.

\begin{corollary}\label{coro:kar22-projected-finited-variation}
  Let the assumption of Lemma~\ref{lem:kar22-admissible-matrix-sense}
  hold. Then for every $d\in\MN$ we have
  \begin{align}\label{eq:kar22-second-moment-mu-d}
    \int_{\MS_{d}^{+}\setminus\set{0}}(\norm{\xi}_{d}\vee \norm{\xi}_{d}^{2})\,\tilde{M}_{d}(\D\xi)<\infty.
  \end{align}
  Moreover, for every $d\in\MN$ and $\mu$ as in
  \cref{item:kar22-affine-kernel} we have
  \begin{align}\label{eq:kar22-finite-variation-projected}
    \int_{\cHpluso}\norm{\bP_{d}(\xi)}\frac{\mu(\D\xi)}{\norm{\xi}^{2}}<\infty,\quad
    \forall\,d\in\MN.
  \end{align}
\end{corollary}
\begin{proof}
  From Proposition~\ref{prop:kar22-Galerkin-explicit} it follows that
  $(0,\tilde{c}_{d},\tilde{D}_{d},0,0,\tilde{m}_{d},\tilde{M}_{d})$ is an
  admissible parameter set as in~\cite[Definition 2.3]{CFMT11}. It then
  follows from \cite[Theorem 3.12]{May12} (which proves that the
  state-dependent jump measure $\tilde{M}_{d}(x,\D\xi)$ is of
  finite-variation) implies that for all $d\in\MN$ the
  $\MS_{d}^{+}$-valued measure $\tilde{M}_{d}$ satisfies
  $\int_{\set{\xi\in\MS_{d}^{+}\colon 0<\norm{\xi}_{d}\leq
      1}}\norm{\xi}_{d}\tilde{M}_{d}(\D\xi)<\infty$, which by the definition of
  $\tilde{M}_{d}(\D\xi)$ implies $\int_{\set{\xi\in\cHplus\colon
      0<\norm{\xi}\leq
      1}}\norm{\bP_{d}(\xi)}\frac{\mu(\D\xi)}{\norm{\xi}^{2}}<\infty$, which
  yields~\eqref{eq:kar22-finite-variation-projected}. 
\end{proof}
   
\begin{remark}\label{rem:kar22-pettis-centering}
  \begin{enumerate}
  \item[i)] Note that from~\eqref{eq:kar22-finite-variation-projected} we conclude
    that the state-dependent jump-measure $M(x,\D\xi)=\norm{\xi}^{-2}\langle
    x,\mu(\D\xi)\rangle$ is of finite-variation
    in every direction $\be_{i,j}$, for $i\leq j\in\MN$, and in every direction
    $v\in\cHplus$ with at most finitely many non-zero coordinates. However, in
    contrast to the finite-dimensional case in $\MS_{d}^{+}$,
    see~\cite{May12}, this in general does not imply that $M(x,\D\xi)$ is
    of finite-variation,
    i.e. $\int_{\cHpluso}\!\norm{\xi}M(x,\D\xi)\!<\!\infty $ ($\,\forall\,
    x\in\cHplus$). Indeed, due to
    the infinite-dimensionality of $\cH$ there are ``infinite many
    directions'', in each of which the jumps evolve with finite-variation, but
    in sum, over all coordinates, the variation could be infinite, see Section~\ref{sec:kar22-regul-affine-proc}. 
  \item[ii)] The situation described in i) is a typical, although not necessary, infinite-dimensional
    phenomenon. Indeed, let $V$ be an infinite-dimensional Banach space and $D_{0}\subseteq V$, then
    the question whether $\int_{D_{0}}\langle
    \xi,u\rangle_{V^{*}}\,\nu(\D\xi)<\infty$  for all $u\in V^{*}$, i.e. the
    Pettis integrability on $D_{0}$, implies
    $\int_{D_{0}}\norm{\chi(\xi)}_{V}\,\nu(\D\xi)$, i.e. the Bochner
    integrability on $D_{0}$, where $V^{*}$ denotes the Banach dual of $V$
    with dual pairing $\langle \cdot,\cdot\rangle_{V^{*}}$, also depends on
    the space $V$. In case of
    Hilbert-Schmidt operators, $D_{0}=\set{\xi\in\cHplus\colon 0<\norm{\xi}\leq
      1}$ and $\nu(\D\xi)=\nu(x,\D\xi)$ this implication does not
    hold true. In contrast, in an analogous situation on the space of
    trace-class operators the above implication does hold, see~\cite{PR06}.  
  \end{enumerate}
\end{remark}  
\subsection{Proof: Existence of finite-rank operator-valued affine processes}\label{sec:kar22-finite-rank-affine}

For every $d\in\MN$, let $(\tilde{X}^{d},(\tilde{\MP}^{d}_{x})_{x\in\MS_{d}^{+}})$
be the $\MS_{d}^{+}$-valued affine process given by
Proposition~\ref{prop:kar22-Galerkin-explicit}. More precisely, let
$\tilde{X}^{d}$ be a version with paths in $\Omega=D(\MRplus,\MS_{d}^{+})$ and
denote by $\tilde{\MP}^{d}_{x}$ the law of $X^{d}$, defined on $\cB(\Omega)$,
given $\tilde{X}_{0}^{d}=x\in\MS_{d}^{+}$. Moreover, let us denote by
$(\tilde{\cF}_{t}^{d})_{t\geq 0}$ the natural filtration of the process
$\tilde{X}^{d}$. By identifying the cones $\MS_{d}^{+}$ and $\cH_{d}^{+}$
under the mapping $i_{d}^{-1}$, we define the process $X^{d}=(X^{d}_{t})_{t\geq 0}$ as  
\begin{align*}
  X^{d}_{t}\df
  i_{d}^{-1}(\tilde{X}_{t}^{d})=\Phi_{d}\circ\tilde{X}_{t}^{d}\circ\Phi_{d}^{-1},\quad
  t\geq 0.  
\end{align*}
Note that the process $(X^{d}_{t})_{t\geq 0}$ has paths in
$D(\MRplus,\cHplus_{d})$ and the law of $X^{d}$ is given by the push-forward
measure $(i_{d}^{-1})_{*}\tilde{\MP}_{x}$ for $x\in\MS_{d}^{+}$, where we
understand that $i_{d}^{-1}$ acts pointwise on the functions in
$D(\MRplus,\MS_{d}^{+})$ such that
$i_{d}^{-1}(D(\MRplus,\MS_{d}^{+}))=D(\MRplus,\cHplus_{d})$. Moreover, we see that
$D(\MRplus,\cHplus_{d})\subseteq D(\MRplus,\cHplus)$ for all $d\in\MN$, see
~\cite[Remark 4.5]{Jak86}. For every $x\in\cHplus$ we define the
measure $\MP_{x}^{d}$ on $D(\MRplus,\cHplus)$ as
\begin{align*}
\MP^{d}_{x}(A)=(i_{d}^{-1})_{*}\tilde{\MP}^{d}_{i_{d}(\bP_{d}(x))}(A\cap
  D(\MRplus,\cHplus_{d})),\quad A\in\cB(D(\MRplus,\cHplus)).  
\end{align*}
Note that
$\MP^{d}_{x}\big(X_{0}^{d}=\bP_{d}(x)\big)=\tilde{\MP}^{d}_{i_{d}(\bP_{d}(x))}\big(\tilde{X}_{0}^{d}=i_{d}(\bP_{d}(x))\big)=1$
and the process $(X^{d},(\MP_{x}^{d})_{x\in\cHplus})$ is again a Markov
process on the ambient space $\cHplus$ with respect to its natural filtration
$\MF^{d}=(\cF_{t}^{d})_{t\geq 0}$ and we set
$\cF^{d}=\cF_{\infty}^{d}$. Moreover, for every $x\in\cHplus$ and Markov
process $X^{d}$, denoting by $\mathcal{N}^{d}_{x}$ the collection of all $\MP_{x}^{d}$-null sets of
$\cF^{d}$, we define $\bar{\cF}_{t}\df\cF_{t}\vee \mathcal{N}_{x}^{d}$ for
every $t\geq 0$ and set $\bar{\MF}^{d}\df (\bar{\cF}_{t})_{t\geq 0}$, i.e. $\bar{\MF}^{d}$ is the usual augmented
filtration of $X^{d}$ and the process $X^{d}$ is still a Markov process with
respect to $\bar{\MF}^{d}$. In addition to that, we prove in the following
proposition that $(X^{d},(\MP_{x}^{d})_{x\in\cH})$ satisfies an affine
transform formula associated with the Galerkin approximations
in~\eqref{eq:kar22-Riccati-Galerkin-phi}-\eqref{eq:kar22-Riccati-Galerkin-psi}
and the canonical process of $\MP_{x}^{d}$ on $D(\MRplus,\MS_{d}^{+})$ is a
semimartingale with respect to the stochastic basis
$(\Omega,\bar{\cF}^{d},\bar{\MF}^{d},\MP_{x}^{d})$.  

\begin{proposition}\label{prop:kar22-embedding-affine}
  Let $(b,B,m,\mu)$ be an admissible parameter set and for $d\in\MN$
  let $(b_{d},B_{d},m_{d},\mu_{d})$ and $M_{d}$ be as in
  Definition~\ref{def:kar22-admissible-Galerkin}. Then for every $d\in\MN$ the
  process $(X^{d},(\MP_{x})_{x\in\cHplus})$ defined as above is a Markov
  process on $\cHplus$ such that for every $x\in \cHplus$ we have  
  \begin{align}\label{eq:kar22-affine-Galerkin-1}
    \EXspec{\MP_{x}^{d}}{\E^{-\langle X^{d}_{t},
    \bP_{d}(u)\rangle}}=\E^{-\phi_{d}(t,\bP_{d}(u))-\langle
    \bP_{d}(x),\psi_{d}(t,\bP_{d}(u))\rangle},\quad t\geq 0,\, u\in\cHplus_{d},  
  \end{align}
  for $\big(\phi_{d}(\cdot,\bP_{d}(u)),\psi_{d}(\cdot,\bP_{d}(u))\big)$ the
  unique solution
  of~\eqref{eq:kar22-Riccati-Galerkin-phi}-\eqref{eq:kar22-Riccati-Galerkin-psi}.
  Moreover, for every $x\in\cHplus$ we have
  \begin{align}\label{eq:kar22-conservative-semi-Galerkin}
    \MP_{x}^{d}(\set{X_{t}^{d}\in\cHplus_{d}\colon\, t\geq 0})=1,  
  \end{align}
  and the canonical process of $\MP_{x}^{d}$ on $D(\MRplus,\cHplus)$, still
  denoted by $(X_{t}^{d})_{t\geq 0}$, is a semimartingale with respect to the
  stochastic basis $(\Omega,\bar{\cF}^{d},\bar{\MF}^{d},\MP_{x}^{d})$ whose
  semimartingale characteristics $(A^{d},C^{d},\nu^{d})$, with respect to
  $\chi$, are given by:   
  \begin{align*}
    A_{t}^{d}&=\int_{0}^{t}b_{d}+B_{d}(X_{s}^{d})\D s,\quad
               C^{d}_{t}&=0,\quad \nu^{d}(\D t,\D\xi)&= \big(m_{d}(\D\xi)+M_{d}(X_{t}^{d},\D\xi)\big)\D t.
  \end{align*}
\end{proposition}
\begin{proof}
  
  Let $d\in\MN$, $x\in\cHplus$ and let $(\tilde{X}^{d}_{t})_{t\geq 0}$ be the
  unique affine process on $\MS_{d}^{+}$ associated with the parameter set
  $(0,\tilde{c}_{d},\tilde{D}_{d},0,0,\tilde{m}_{d},\tilde{M}_{d})$
  and such that $\tilde{X}^{d}_{0}=i_{d}(\bP_{d}(x))$. For $u\in\cHplus$ we have 
  \begin{align*}
    \EXspec{\MP_{x}^{d}}{\E^{-\langle \bP_{d}(u), X_{t}^{d}\rangle}}&=\EXspec{\MP_{x}^{d}}{\E^{-\langle
                                                              \bP_{d}(u),i_{d}^{-1}(\tilde{X}_{t}^{d})\rangle}}\\
                                                            &=\EXspec{\tilde{\MP}_{i_{d}(\bP_{d}(x))}^{d}}{\E^{-\langle i_{d}\bP_{d}(u),\tilde{X}_{t}^{d}\rangle_{d}}}\\
                                                            &=\E^{-\tilde{\phi}_{d}(t,i_{d}(\bP_{d}(u)))-\langle
                                                              i_{d}\bP_{d}(x),\tilde{\psi}_{d}(t,i_{d}(\bP_{d}(u)))\rangle_{d}}\\
                                                            &=\E^{-\tilde{\phi}_{d}(t,i_{d}(\bP_{d}(u)))-\langle
                                                              \bP_{d}(x)
                                                              ,i_{d}^{-1}\tilde{\psi}_{d}(t,i_{d}(\bP(u)))\rangle}. 
  \end{align*}
  This proves that the process $X^{d}_{t}$ satisfies the affine transform formula with
  functions $\tilde{\phi}_{d}(t,i_{d}(\bP_{d}(u)))$ and
  $i_{d}^{-1}(\tilde{\psi}_{d}(t,i_{d}(\bP_{d}(u)))$. Therefore, in order to
  prove~\eqref{eq:kar22-affine-Galerkin}, it is left to show that
  $\big(\phi_{d}(\cdot,\bP_{d}(u)),\psi_{d}(\cdot,\bP_{d}(u))\big)$, the unique
  solution
  of~\eqref{eq:kar22-Riccati-Galerkin-phi}-\eqref{eq:kar22-Riccati-Galerkin-psi}
  coincides with the function
  $\big(\tilde{\phi}_{d}(\cdot,i_{d}(\bP_{d}(u))),i_{d}^{-1}(\tilde{\psi}_{d}(\cdot,i_{d}(\bP_{d}(u)))\big)$. For
  this, let us again consider $K\colon\cH\times\cH\to \MR$ given by $K(u,v)\df\E^{-\langle
    u,v\rangle}-1+\langle \chi(u),v\rangle$ and for every $u\in\cHplus$, we set
  $\tilde{u}_{d}\df i_{d}(\bP_{d}(u))$, then we see that for all $t\geq 0$ and
  $u\in\cHplus$ the function
  $i_{d}^{-1}(\tilde{\psi}_{d}(\cdot,i_{d}(\bP_{d}(u))))$ satisfies the
  following equation: 
  \begin{align*}
    \frac{\partial \,i_{d}^{-1}(\tilde{\psi}_{d}(t,\tilde{u}_{d}))}{\partial
    t}&= i_{d}^{-1}(\tilde{R}_{d}(\tilde{\psi}_{d}(t,\tilde{u}_{d})))\\
    &=
       i_{d}^{-1}(\tilde{B}_d^{*}(\tilde{\psi}_{d}(t,\tilde{u}_{d})))-\int_{\MS_{d}^{+}\setminus\set{0}}K\big(\xi,
      \tilde{\psi}_{d}(t,i_{d}(\bP_{d}(u)))\big)i_{d}^{-1}(\tilde{M}_{d}(\D\xi)) \\
      &=B_{d}^{*}(i_{d}^{-1}(\tilde{\psi}(t,\tilde{u}_{d}))-\int_{\cHplus_{d}\setminus\set{0}}K\big(
                                           \xi, i_{d}^{-1}(\psi_{d}(t, \tilde{u}_{d}))\big)\,M_{d}(\D\xi)),
  \end{align*}
  and
  $i_{d}^{-1}(\tilde{\psi}_{d}(0,\tilde{u}_{d}))=\tilde{u}_{d}=i_{d}^{-1}(i_{d}\bP_{d}(u))=\bP_{d}(u)$. But
  since~\eqref{eq:kar22-Riccati-Galerkin-psi} is uniquely solved by
  $\psi_{d}(\cdot,\bP_{d}(u))$ we conclude that
  $\psi_{d}(\cdot,\bP_{d}(u))=i_{d}^{-1}(\tilde{\psi}_{d}(\cdot,i_{d}(\bP_{d}(u))))$. Similarly,
  for $\tilde{\phi}_{d}(\cdot,\tilde{u}_{d})$ we find 
  \begin{align*}
    \frac{\partial \tilde{\phi}_{d}(t,\tilde{u}_{d})}{\partial t}&= \tilde{F}_{d}(\tilde{\psi}_{d}(t,\tilde{u}_{d}))\\                                                                        
                                                                     &=\langle \tilde{b}_{d},\tilde{\psi}_{d}(t,\tilde{u}_{d})\rangle-\int_{\MS_{d}^{+}\setminus\set{0}}K\big(\xi,\tilde{\psi}_{d}(t,\tilde{u}_{d})\big)\,\tilde{m}_{d}(\D\xi)\\
                                         &=\langle
                                           b_{d},i_{d}^{-1}\tilde{\psi}_{d}(t,\tilde{u}_{d})\rangle-\int_{\cHplus_{d}\setminus\set{0}}K\big(\xi,i_{d}^{-1}\tilde{\psi}_{d}(t,\tilde{u}_{d})\big)\,m_{d}(\D\xi), 
  \end{align*}
  and $\tilde{\phi}_{d}(0,\tilde{u}_{d})=0$. Again by the uniqueness of the
  solution to~\eqref{eq:kar22-Riccati-Galerkin-phi} we conclude that
  $\phi_{d}(\cdot,\bP_{d}(u))=\tilde{\phi}_{d}(\cdot,i_{d}(\bP_{d}(u)))$,
  which finally proves~\eqref{eq:kar22-affine-Galerkin}. Moreover, the
  property~\eqref{eq:kar22-conservative-semi-Galerkin} follows from
  Proposition~\ref{prop:kar22-Galerkin-explicit} and 
  \begin{align*}
  \MP_{x}^{d}(\set{X_{t}^{d}\in\cHplus_{d}\colon\, t\geq
    0})=\tilde{\MP}^{d}_{i_{d}(\bP_{d}(x))}(\set{\tilde{X}_{t}^{d}\in\MS^{+}_{d}\colon\, t\geq 0})=1.  
  \end{align*}
  The asserted form of the semimartingale characteristics and the square-integrability
  follows immediately from the analogous property in the matrix-valued case
  and an application of the linear isometric transformation $i_{d}^{-1}$.
\end{proof}

With Proposition~\ref{prop:kar22-embedding-affine} we already have already
shown the
first part of~\cref{item:kar22-embedding-affine-main-1}. In the next
proposition we assert some additional properties of the process $X^{d}$. In particular, we show that
$X^{d}$ solves the martingale problem for $\cG^{d}$, from which we conclude
that the second assertion of~\cref{item:kar22-embedding-affine-main-2} holds true.

\begin{proposition}\label{prop:martingale-problem-d}
  For $d\in\MN$ and $x\in\cH_{d}$, let $X^{d}$ denote the
  affine process on $\cHplus_{d}$ with $X_{0}^{d}=\bP_{d}(x)$ given by
  Proposition~\ref{prop:kar22-embedding-affine}. Then the process $(\bar{J}^{d}_{t})_{t\geq 0}$ given by 
  \begin{align}\label{eq:kar22-decomp-martingale} 
    \bar{J}^{d}_{t}&\df
                     X_{t}^{d}\!-\!\bP_{d}(x)\!-\!\int_{0}^{t}\!\!\big(b_{d}+B_{d}(X_{s}^{d})-\!\int_{\cHplus_{d}\cap\set{\norm{\xi}>
                     1}}\!\!\xi\,(m_{d}(\D\xi)\!+\!M_{d}(X_{s}^{d},\D\xi)\big)\D s, 
  \end{align}
  is a square-integrable martingale on $\cH_{d}$. Moreover,
  we define for every function $f\in\dom(\cG^{d})\df\lin\set{\E^{-\langle\cdot, \bP_{d}(u)\rangle},\,\big\langle
    \cdot,\bP_{d}(u)\big\rangle,\,\big\langle
    \cdot,\bP_{d}(u)\big\rangle^{2}\colon u\in \cHplus}$ the operator $\cG^{d}$ as
  \begin{align}\label{eq:Gd-derivative-def}
    \cG^{d}f(x)=\langle b_{d}+B_{d}(x),f'(x)\rangle+\!\int_{\cHplus_{d}\setminus\set{0}}\!\!\big(f(x+\xi)\!-\!f(x)\!-\!\langle
    \chi(\xi),f'(x)\rangle\big)\,\nu(x,\D\xi), 
  \end{align}
  where $f'(x)$ denotes the first derivative of $f$ at $x\in\cHplus$, i.e. we
  have $(\E^{-\langle
    \cdot,\bP_{d}(u)\rangle})'=- \bP_{d}(u) \E^{-\langle \cdot,\bP_{d}(u)\rangle}$, $(\langle\cdot, \bP_{d}(u)\rangle)'= \bP_{d}(u)$ and $(\langle \cdot,\bP_{d}(u)\rangle^{2})'= 2\bP_{d}(u)\langle
  \cdot , \bP_{d}(u)\rangle$ for every $u\in\cHplus$. Then for all
  $f\in\dom(\cG^{d})$ the process
  \begin{align}\label{eq:Gd-martingale}
    \Big(f(X^{d}_{t})-f(\bP_{d}(x))-\int_{0}^{t}\cG^{d}f(X_{s})\,\D s
    \Big)_{t\geq 0}, 
  \end{align}
  is a real-valued martingale. 
\end{proposition}
\begin{proof}
  Note first that we can extend the operator $B_{d}$ to $\cH$ by setting $B_{d}(u)\df B_{d}(\bP_{d}(u))$ for
  $u\in\cHplus$ and the measures $m_{d}$ and $\mu_{d}$ to $\cB(\cHpluso)$ by setting $m_{d}(A)=m_{d}(A\cap
  (\cHplus_{d}\setminus\set{0}))$ for $A\in\cB(\cHpluso)$ and analogously for
  $\mu_{d}$. We denote the extended generators again by $B_{d}$, $m_{d}$ and
  $\mu_{d}$ and note that $(b_{d},B_{d},m_{d},\mu_{d})$ satisfies the
  conditions in Definition~\ref{def:kar22-admissible}. The
  representation~\eqref{eq:kar22-decomp-martingale} thus follows
  from~\cite[Proposition 2.4]{CKK22b}. Moreover, we see that the operator
  $\cG^{d}$ defined in~\eqref{eq:Gd-derivative-def} on $\dom(\cG^{d})$
  coincides with the \emph{weak generator} as introduced in~\cite[Definition
  2.1]{CKK22b} and that the processes in~\eqref{eq:Gd-martingale} are
  real-valued martingales for all $f\in\dom(\cG^{d})$ thus follows
  from~\cite[Proposition 2.5]{CKK22b}. Note, in particular that $\cG^{d}$
  applied to $\E^{-\langle \cdot,u\rangle}$ evaluated at $x\in\cHplus$ can be
  computed as
  \begin{align*}
    \cG^{d}\E^{-\langle \cdot,u\rangle}(x)&=\Big(-\langle
                                            b_{d}+B_{d}(x),\bP_{d}(u)\rangle\\
                                          &\qquad +\!\int_{\cHplus_{d}\setminus\set{0}}\!\!\big(
                                            \E^{-\langle \xi,\bP_{d}(u)\rangle}-
                                            1+\langle
                                            \chi(\xi),\bP_{d}(u)\rangle\big)\,\nu(x,\D\xi)\Big)\E^{-\langle x,\bP_{d}(u)\rangle}\\
                                          &=(-F_{d}(u)-\langle
                                            x,R_{d}(u)\rangle)\E^{-\langle
                                            x,\bP_{d}(u)\rangle},   
  \end{align*}
  and we see that $\cG^{d}$ defined in~\eqref{eq:Gd-derivative-def}
  coincides with $\cG^{d}$ in~\eqref{eq:kar22-G-d-operator}, which also explains our notation.
\end{proof}

\section{Tightness and weak convergence of finite-rank affine
  processes}\label{sec:kar22-tightness-weak-convergence}

Let $(b,B,m,\mu)$ be an admissible parameter set and for every $d\in\MN$ let
$X^{d}$ denote the associated affine finite-rank operator-valued process given
by Proposition~\ref{prop:kar22-embedding-affine}. In this section we study
the tightness and weak-convergence of the sequence $(X^{d})_{d\in\MN}$ on the
space $D(\MRplus,\cHplus)$ equipped with the Skorohod topology. More
precisely, for every $x\in\cHplus$ we consider the sequence
$(\MP_{x}^{d})_{d\in\MN}$ of laws of $X^{d}$, given that
$X_{0}^{d}=\bP_{d}(x)$, defined on the Borel-$\sigma$-algebra
$\cB(D(\MRplus,\cHplus))$, and study its weak convergence as $d\to\infty$. For
this, we shall first prove that the sequence of laws $(\MP_{x}^{d})_{d\in\MN}$ is
tight on $\cB(D(\MRplus,\cHplus))$, whenever
Assumption~\ref{assump:kar22-compact-embedding} is satisfied. This we prove in
Section~\ref{sec:kar22-tightn-finite-rank}. Subsequently, in
Section~\ref{sec:kar22-weak-conv-finite}, we prove weak convergence of
$(\MP_{x}^{d})_{d\in\MN}$ to a unique probability measure $\MP_{x}$ on
$\cB(D(\MRplus,\cHplus))$, the canonical process of which turns out to be the
desired affine process $(X,\MP_{x})$ on $\cHplus$ and we prove the remaining assertions
of Theorem~\ref{thm:kar22-main-convergence}.  

\subsection{Tightness}\label{sec:kar22-tightn-finite-rank}
To prove the tightness of the sequence $(\MP_{x}^{d})_{d\in\MN}$ we use the
\emph{Aldous criterion} in~\cite[Theorem 2.2.2]{JM86}, which we shall recall
in the beginning of the proof of Proposition~\ref{prop:kar22-tightness}
below. We first need the following lemma:  

\begin{lemma}\label{lem:kar22-tightness-square-bound}
  Let $x\in\cHplus$, $T>0$ and for every $d\in\MN$ denote by
  $(\bar{J}_{t}^{d})_{t\geq 0}$ the square-integrable martingale given
  by~\eqref{eq:kar22-decomp-martingale}. Then there exists a constant
  $K_{T}\geq 0$ such that the following inequalities hold true: 
  \begin{align}
    \EXspec{\MP_{x}^{d}}{\sup_{0\leq t\leq T}\norm{X_{t}^{d}}^{2}}&\leq K_{T}(1+\norm{x}^{2}),\label{eq:kar22-tightness-square-bound-Xd}\\
    \EXspec{\MP_{x}^{d}}{\sup_{0\leq t\leq T}\norm{ \bar{J}_{t}^{d}}^{2}}&\leq K_{T}(1+\norm{x}^{2}).\label{eq:kar22-tightness-square-bound-Md} 
  \end{align}
  Moreover, $K_{T}$ can be chosen independently of $d\in\MN$.
\end{lemma}
\begin{proof}
  Let $d\in\MN$ and $(b_{d},B_{d},m_{d},\mu_{d})$ and $M_{d}$ be as in
  Definition~\ref{def:kar22-admissible-Galerkin}. Then define
  $\hat{b}_{d}\df b_{d}+\int_{\cHplus_{d}\cap\set{\norm{\xi}>1}}\xi\,m_{d}(\D\xi)$ and the
  function $\hat{B}_{d}\colon\cHplus_{d}\to\cH_{d}$ by   
  \begin{align*}
    \hat{B}_{d}(u)\df B_{d}(u)+\int_{\cHplus_{d}\cap\set{\norm{\xi}>1}}\xi\,
    \langle u, M_{d}(\D\xi)\rangle,\quad u\in\cH_{d}.   
  \end{align*}
  By Proposition~\ref{prop:martingale-problem-d} we have
  $X^{d}_{t}=\bP_{d}(x)+H_{t}^{d}+\bar{J}_{t}^{d}$ for every $t\in[0,T]$, where
  $(\bar{J}^{d}_{t})_{0\leq t\leq T}$ denotes the square-integrable martingale
  in~\eqref{eq:kar22-decomp-martingale} on $[0,T]$ and we write $(H_{t}^{d})_{0\leq t\leq T}$
  for the finite-variation process given by 
  \begin{align}\label{eq:kar22-Fd}
    H^{d}_{t}\df\int_{0}^{t}\big(\hat{b}_{d}+\hat{B}_{d}(X_{s}^{d})\big)\D
    s,\quad 0\leq t\leq T.
  \end{align}
  We therefore obtain
  \begin{align}\label{eq:kar22-tightness-square-bound-1}
    \EXspec{\MP_{x}^{d}}{\sup_{0\leq t\leq T}\norm{X_{t}^{d}}^{2}}\leq 3 \norm{\bP_{d}(x)}^{2}\!+\!3
    \EXspec{\MP_{x}^{d}}{\sup_{0\leq t\leq
    T}\norm{H_{t}^{d}}^{2}}\!+\!3\EXspec{\MP_{x}^{d}}{\sup_{0\leq t\leq T}\norm{\bar{J}^{d}_{t}}^{2}}.
  \end{align}
  Inserting~\eqref{eq:kar22-Fd} into the second term on the right-hand side
  of~\eqref{eq:kar22-tightness-square-bound-1} yields
  \begin{align}\label{eq:kar22-tightness-square-bound-2}
    \EXspec{\MP_{x}^{d}}{\sup_{0\leq t\leq T}\norm{H_{t}^{d}}^{2}}&\leq 2 T^{2}\norm{\hat{b}_{d}}^{2}+
                                                              2\norm{\hat{B}_{d}}^{2}_{\cL(\cH_{d})}\int_{0}^{T}\EXspec{\MP_{x}^{d}}{\norm{X_{s}^{d}}^{2}}\D
                                                              s\nonumber \\
                                                            &\leq 2 T^{2}\norm{\hat{b}}^{2}+
                                                              2\norm{\hat{B}}^{2}_{\cL(\cH)}\int_{0}^{T}\EXspec{\MP_{x}^{d}}{\norm{X_{s}^{d}}^{2}}\D
                                                              s,
  \end{align}
 where the latter inequality for $\hat{b}\df
 b+\int_{\cHplus\cap\set{\norm{\xi}>1}}\xi \,m(\D\xi)$ and linear function $\hat{B}(\cdot)\df
 B(\cdot)+\int_{\cHplus\cap\set{\norm{\xi}\geq 1}}\xi\,\langle
 \,\cdot\,,M(\D\xi)\rangle$ holds by Remark~\ref{rem:kar22-existence-integral}.
 For the second term in~\eqref{eq:kar22-tightness-square-bound-1}, we recall from \cite[Theorem 20.6]{Met82} that  
 \begin{align}\label{eq:kar22-doob-ineqaulity}
   \EXspec{\MP_{x}^{d}}{\sup_{0\leq t\leq T}\norm{\bar{J}^{d}_{t}}^{2}}\leq 4 \EXspec{\MP_{x}^{d}}{\big< \bar{J}^{d}\big>_{T}},
 \end{align}
 where we denote by $\big(\big< \bar{J}^{d} \big>_{t}\big)_{0\leq t\leq T}$ the angle-bracket process of the
 square-integrable martingale $(\bar{J}^{d}_{t})_{0\leq t\leq T}$. Now, let
 $(\be_{i,j})_{i\leq j\in\MN}$ be the same orthonormal basis of $\cH$ that we
 used throughout this section. For $i\leq j\in\MN$ we set
 $\bar{J}_{t}^{(i,j),d}\df \langle \bar{J}^{d}_{t},\be_{i,j}\rangle$ and denote by
 $\big(\big<\bar{J}^{(i,j),d}\big>_{t}\big)_{t\geq 0}$ the unique real-valued increasing process
 such that 
 \begin{align*}
   \big((\bar{J}^{(i,j),d})^{2}-\big<\bar{J}^{(i,j),d}\big>\big)_{0\leq t\leq T}, 
 \end{align*}
 is a martingale. Moreover, as in \cite[Section 20]{Met82}, we denote by $\big(\big<
 \bar{J}^{d}\big>_{t}\big)_{0\leq t\leq T}$ the unique predictable and increasing
 process such that $\big(\norm{\bar{J}^{d}}^{2}-\big<\bar{J}^{d}\big>\big)_{0\leq t\leq T}$ is a
 martingale. Note that $\big<\bar{J}^{d}\big>_{t}=\sum_{i\leq
   j}^{d}\big<\bar{J}^{(i,j),d}\big>_{t}$ for every $0\leq t\leq T$ and it is thus
 left to compute the form of the processes
 $\big(\big<\bar{J}^{(i,j),d}\big>_{t}\big)_{0\leq t\leq T}$ for $1\leq i\leq
 j\leq d$. By an application of the \emph{Carr\'e-du-champs formula}, see
 e.g.~\cite[Lemma 3.1.3]{JM86}, we see that 
 \begin{align}
   \big<\bar{J}^{(i,j),d}\big>_{t}&=\int_{0}^{t}\cG^{d}\langle
                              X_{s}^{d},\be_{i,j}\rangle^{2}-2\langle
                              X_{s}^{d},\be_{i,j}\rangle\cG^{d}\langle
                              X_{s}^{d},\be_{i,j}\rangle \D s,\quad 0\leq
                              t\leq T,\label{eq:kar22-carre-du-champs}
 \end{align}
 where $\cG^{d}$ is the operator in~\eqref{eq:Gd-derivative-def}, where by
 linearity we extend $\cG^{d}$ to the set $\lin(\set{\langle
   \cdot,\bP_{d}(u)\rangle,\langle \cdot,\bP_{d}(u)\rangle^{2}\colon
   u\in\cH})$, see also~\cite[Lemma 3.9 and Proposition 4.17]{CKK22a}, and we
 thus obtain
  \begin{align}
    \cG^{d}\langle x,\be_{i,j}\rangle &=\langle
                                        b_{d}\!+\!B_{d}(x),\be_{i,j}\rangle+\!\int_{\cHplus\cap\set{\norm{\xi}>1}}\!\langle
                                        \xi,\be_{i,j}\rangle\,\big(m_{d}(\D\xi)\!+\!\langle
                                        x,M_{d}(\D\xi)\rangle\big),\label{eq:kar22-G-d-operators-1}\\
      \cG^{d}\langle x,\be_{i,j}\rangle^{2}&=\!\int_{\cHpluso}\!\langle
                                             \xi,\be_{i,j}\rangle^{2}\,\big(m_{d}(\D\xi)\!+\!\langle
                                             x,
                                             M_{d}(\D\xi)\rangle\big)+2\langle
                                             x,\be_{i,j}\rangle \cG^{d}\langle
                                             x,\be_{i,j}\rangle.\label{eq:kar22-G-d-operators-2}    
  \end{align}
 Inserting~\eqref{eq:kar22-G-d-operators-1}
 and~\eqref{eq:kar22-G-d-operators-2} into~\eqref{eq:kar22-carre-du-champs}
 yields  
 \begin{align*}
   \big< \bar{J}^{(i,j),d}\big>_{t}&=\int_{0}^{t}\int_{\cHpluso}\langle \xi,\be_{i,j}\rangle^{2}\,\big(m_{d}(\D\xi)+\langle
                               X_{s}^{d}, M_{d}(\D\xi)\rangle\big)\,\D s,\quad
                                     t\in[0,T].
 \end{align*}
Now, since we have
\begin{align*}
  \sum_{i\leq j}^{d}\Big(\int_{\cHplus_{d}\setminus\set{0}}\langle \xi,\be_{i,j}\rangle^{2}\,m_{d}(\D\xi)\Big)&=\int_{\cHplus_{d}\setminus\set{0}}\norm{\xi}^{2}\,m_{d}(\D\xi)\leq \int_{\cHpluso}\norm{\xi}^{2}\,m(\D\xi),  
\end{align*}
and moreover for every $s\in [0,t]$
\begin{align*}
  \sum_{i\leq j}^{d}\Big(\int_{\cHplus_{d}\setminus\set{0}}\langle \xi,
  \be_{i,j}\rangle^{2}\langle X_{s}^{d},
  M_{d}(\D\xi)\rangle\Big)&=\int_{\cHplus_{d}\setminus\set{0}}\norm{\xi}^{2}\langle
                            X_{s}^{d}, M_{d}(\D\xi)\rangle\\
                          &= \int_{\cHpluso}\frac{\norm{\bP_{d}(\xi)}^{2}}{\norm{\xi}^{2}}\langle
                         X_{s}^{d}, \bP_{d}(\mu(\D\xi))\rangle\\
                          &\leq \langle X_{s}^{d}, \mu(\cHpluso)\rangle,
\end{align*}
we conclude that for every $d\in\MN$ and $0\leq t\leq T$ the following
inequality holds
\begin{align*}
  \big< \bar{J}^{d}\big>_{t}&=\sum_{i\leq
                        j}^{d}\big<\bar{J}^{(i,j),d}\big>_{t}\leq \int_{0}^{t}\Big(\int_{\cHpluso}\norm{\xi}^{2}\,m(\D\xi)+\langle
                        X_{s}^{d},\mu(\cHpluso)\rangle \Big)\D s.
\end{align*}
From this it follows that
\begin{align}
  \EXspec{\MP_{x}^{d}}{\big<\bar{J}^{d}\big>_{T}}&\leq T\int_{\cHpluso}\!\norm{\xi}^{2}\,m(\D\xi)+
                                       \norm{\mu(\cHpluso)}\Big(\int_{0}^{T}\!\EXspec{\MP_{x}^{d}}{\norm{X_{s}^{d}}}\D s\Big) \nonumber\\
                                     &\leq T\int_{\cHpluso}\!\norm{\xi}^{2}\,m(\D\xi)+
                                       \norm{\mu(\cHpluso)}\Big(\int_{0}^{T}\!\EXspec{\MP_{x}^{d}}{1+\norm{X_{s}^{d}}^{2}}\D
                                       s\Big),\label{eq:kar22-jump-part-doob} 
\end{align}
and hence inserting~~\eqref{eq:kar22-jump-part-doob}
and~\eqref{eq:kar22-tightness-square-bound-2} back into~\eqref{eq:kar22-tightness-square-bound-1} gives
\begin{align*}
  \EXspec{\MP_{x}^{d}}{\sup_{0\leq t\leq T}\norm{X_{t}^{d}}^{2}}&\leq \norm{\bP_{d}(x)}^{2}\!+\!6\big(\norm{\hat{B}}_{\cL(\cH)}\!+\!12\norm{\mu(\cHplus\!\setminus\!\set{0})}\big)\!\int_{0}^{T}\!\!\EXspec{\MP_{x}^{d}}{\norm{X_{s}^{d}}^{2}}\D
                                                              s\\
                                                          &\quad+6T^{2}\norm{\hat{b}}^{2}+12T\big(\int_{\cHpluso}\norm{\xi}^{2}\,m(\D\xi)+T\norm{\mu(\cHpluso)}\big).
\end{align*}
Therefore setting $K_{1,T}=6T^{2}\norm{\hat{b}}^{2}+12T\big(\int_{\cHpluso}\norm{\xi}^{2}\,m(\D\xi)+T\mu(\cHpluso)\big)$
  and $K_{2}=6\big(\norm{\hat{B}}_{\cL(\cH)}+12\mu(\cHpluso)\big)$ (where we
  note that $K_{1,T}$ and $K_{2}$ do not depend on $d\in\MN$) and by applying
  Gronwall's inequality we find that   
  \begin{align*}
    \EXspec{\MP_{x}^{d}}{\sup_{0\leq t\leq T}\norm{X_{t}^{d}}^{2}}&\leq
                                                              \E^{K_{2}T}(K_{1,T}+\norm{\bP_{d}(x)}^{2})\leq
                                                              \tilde{K}_{1,T}(1+\norm{x}^{2}),  
  \end{align*}
  for some $\tilde{K}_{1,T}$, independent of $d\in\MN$, which proves
  inequality~\eqref{eq:kar22-tightness-square-bound-Xd}. Inserting, this back
  into~\eqref{eq:kar22-jump-part-doob}
  yields~\eqref{eq:kar22-tightness-square-bound-Md} for a suitable
  $\tilde{K}_{2,T}$ and choosing $K_{T}=\max(\tilde{K}_{1,T},\tilde{K}_{2,T})$
  proves the assertion.
\end{proof}

Recall the Hilbert space $(\cV,\langle \cdot,\cdot\rangle_{\cV})$
from~\eqref{eq:kar22-intersection-HS-space} and let
Assumption~\ref{assump:kar22-compact-embedding} be satisfied. Then the
embedding of $(\cV,\langle \cdot,\cdot\rangle_{\cV})$ into
$(\cL_{2}(H),\langle\cdot,\cdot\rangle)$ turns out to be compact as well,
i.e. $\cV\subset\!\subset \cL_{2}(H)$, see \cite[Proposition
2.1]{Tem71}. Moreover, we note that $\cH_{d}\subseteq \cV_{0}\cap \cH$ for all
$d\in\MN$, see~\cite{Ros91}. In the next proposition we prove that in
this setting the sequence $(\MP^{d}_{x})_{d\in\MN}$ is tight. 

\begin{proposition}\label{prop:kar22-tightness}
  Let Assumption~\ref{assump:kar22-compact-embedding} be satisfied. Then for
  every $x\in\cHplus$ the sequence $(\MP^{d}_{x})_{d\in\MN}$ of laws of
  $(X^{d})_{d\in\MN}$ is a tight sequence of measures on
  $\cB(\!D(\MRplus,\cHplus)\!)$.  
\end{proposition}
\begin{proof}
  Let $x\in\cHplus$. As mentioned before, we use the \emph{tightness criterion
    from Aldous}, see \cite[Theorem
  2.2.2]{JM86}. For the readers convenience we recall in the following the two
  sufficient conditions implying the tightness of $(\MP^{d}_{x})_{d\in\MN}$:
  \begin{enumerate}
  \item[i)] For every $t\geq 0$ the sequence of laws of
    $(X_{t}^{d})_{d\in\MN}$ form a tight sequence of probability measures on
    $\cB(\cHplus)$, the Borel-$\sigma$-algebra on $\cHplus$. 
  \item[ii)] For every $T>0$, $\varepsilon>0$, $\eta>0$ there exists a
    $\delta>0$ 
    and $N_{0}\in\MN$ such that for every sequence of stopping times
    $(\tau_{d})_{d\in\MN}$ with $\tau_{d}\leq T$ for all $d\in\MN$, we have:  
    \begin{align}\label{eq:kar22-Aldous}
      \sup_{d\geq N_{0}}\sup_{0\leq\theta\leq
      \delta}\MP_{x}^{d}(\norm{X_{\tau_{d}}^{d}-X_{\tau_{d}+\theta}^{d}}\geq
      \eta)\leq \varepsilon.
    \end{align} 
  \end{enumerate}
  We begin with the first condition: Recall that for all
  $d\in\MN$ the processes $X^{d}$ satisfies
  $\MP_{x}^{d}(\set{X_{t}^{d}\in\cHplus_{d}\colon t\geq 0})=1$. In particular,
  for every fixed $t\geq 0$ it holds that
  $\MP_{x}^{d}(X_{t}^{d}\in\cHplus_{d})=1$. Now, note that $\cHplus_{d}\subseteq
  \cV_{0}\cap \cHplus$ for all $d\in\MN$ and since 
  $\cV$ is compactly embedded in $\cL_{2}(H)$ and $\cHplus$ is a closed subset
  of $\cL_{2}(H)$, we see that also $\cV_{0}\cap\cHplus$ is compact in
  $\cHplus$. Hence, we see that $\MP_{x}^{d}(\set{X_{t}^{d}\in
    \cV_{0}\cap\cHplus})=1$ for every $d\in\MN$, which proves the tightness of
  the sequence of laws of $(X_{t}^{d})_{d\in\MN}$. Since $t\geq 0$ was
  arbitrary, we therefore conclude that condition i) is satisfied. We continue with the
  second condition. For this let $T>0$, $\varepsilon>0$, $\eta>0$ and let
  $(\tau_{d})_{d\in\MN}$ be a sequence of stopping times such that
  $\tau_{d}\leq T$ for all $d\in\MN$. As before in the proof of
  Lemma~\ref{lem:kar22-tightness-square-bound} we consider for every $t\geq 0$
  the decomposition $X^{d}_{t}=\bP_{d}(x)+H_{t}^{d}+\bar{J}_{t}^{d}$ into the
  finite variation part $(H_{t}^{d})_{t\geq 0}$ given by~\eqref{eq:kar22-F}
  and the purely-discontinuous martingale part $(\bar{J}^{d}_{t})_{t\geq 0}$
  in~\eqref{eq:kar22-decomp-martingale}. For the finite-variation part we compute
  \begin{align}\label{eq:kar22-tightness-1}
    \EXspec{\MP_{x}^{d}}{\norm{H_{\tau_{d}}^{d}-H_{\tau_{d}+\theta}^{d}}^{2}}&\leq
                                                                         \EXspec{\MP_{x}^{d}}{\norm{\int_{\tau_{d}}^{\tau_{d}+\theta}\big(\hat{b}_{d}+\hat{B}_{d}(X_{s}^{d})\big)\D
                                                                   s}^{2}}\nonumber\\
                                                      &\leq \theta^{2}
                                                        \EXspec{\MP_{x}^{d}}{\sup_{0\leq
                                                        \tau\leq \theta}
                                                        \big(\norm{\hat{b}}+\norm{\hat{B}}_{\cL(\cH)}\norm{X_{\tau_{d}+\tau}^{d}}^{2}\big)}\nonumber\\
                                                      &\leq
                                                        \theta^{2}\big(\norm{\hat{b}}+\norm{\hat{B}}_{\cL(\cH)}K_{T+\theta}(1+\norm{\bP_{d}(x)}^{2})\big),   
  \end{align}
  where in the last inequality we
  used~\eqref{eq:kar22-tightness-square-bound-Xd} and that $\tau_{d}\leq T$ by
  assumption. Similarly, for the martingale part we find 
  \begin{align}\label{eq:kar22-tightness-2}
    \EXspec{\MP_{x}^{d}}{\norm{\bar{J}_{\tau_{d}+\theta}^{d}\!-\!\bar{J}_{\tau_{d}}^{d}}^{2}}&\leq 4
                                                                         \EXspec{\MP_{x}^{d}}{\big<\bar{J}^{d}\big>_{\tau_{d}+\theta}-\big<\bar{J}^{d}\big>_{\tau_{d}}}\nonumber\\
                                                                       &\leq\!
                                                                         4\EXspec{\MP_{x}^{d}}{\!\int_{\tau_{d}}^{\tau_{d}+\theta}\!\Big(\!\int_{\cHpluso}\!\norm{\xi}^{2}\,m(\D\xi)\!+\!\langle
                                                                         X_{s}^{d},\mu(\cHpluso)\rangle
                                                                         \Big)\D
                                                                         s}\nonumber\\
                                                            &\leq\!
                                                              4\theta \EXspec{\MP_{x}^{d}}{
                                                              \sup_{0\leq\tau\leq
                                                              \theta}\Big(\!\int_{\cHpluso}\!\norm{\xi}^{2}\,m(\D\xi)\!+\!\langle
                                                              X_{\tau_{d}+\tau}^{d},\mu(\cHplus\!\setminus\set{0})\rangle
                                                              \Big)}\nonumber\\
                                                            &\leq 
                                                              \!4\theta\big(\!\int_{\cHpluso}\!\!\norm{\xi}^{2}\,m(\D\xi)\!+\!\norm{\mu(\cHpluso)}K_{T\!+\!\theta}(1\!+\!\norm{\bP_{d}(x)}^{2})\big).  
  \end{align}
  By an application of Markov's inequality we thus see that
  \begin{align*}
    \MP_{x}^{d}\big(\norm{X_{\tau_{d}+\theta}^{d}-X_{\tau_{d}}^{d}}>\eta \big)&\leq
                                                                                          \MP^{d}_{x}\big(\norm{H_{\tau_{d}+\theta}^{d}-H_{\tau_{d}}^{d}}+\norm{\bar{J}_{\tau_{d}+\theta}^{d}-\bar{J}_{\tau_{d}}^{d}}>\eta
                                                                                    \big)\nonumber\\
                                                                                  &\leq
                                                                                    \frac{2}{\eta^{2}}\big(\EXspec{\MP_{x}^{d}}{\norm{H_{\tau_{d}+\theta}^{d}-H_{\tau_{d}}^{d}}^{2}}+\EXspec{\MP_{x}^{d}}{\norm{\bar{J}_{\tau_{d}+\theta}^{d}-\bar{J}_{\tau_{d}}^{d}}^{2}}\big),
  \end{align*}
  and therefore by inserting~\eqref{eq:kar22-tightness-1}
  and~\eqref{eq:kar22-tightness-2} we obtain
  \begin{align}\label{eq:kar22-tightness-3}
    \MP_{x}^{d}\big(\norm{X_{\tau_{d}+\theta}^{d}-X_{\tau_{d}}^{d}}>\eta
    \big)&\leq \theta\frac{\hat{K}_{T+\theta}}{\eta^{2}}(1+\norm{\bP_{d}(x)}^{2}),
  \end{align}
  for a $\hat{K}_{T+\theta}$ which is independent of $d\in\MN$ and continuous
  in $\theta$. Moreover, since $\norm{\bP_{d}(x)}\leq \norm{x}$
  for all $d\in\MN$, we find a $\delta>0$ small enough such that
  \begin{align*}
    \sup_{d\geq N_{0}}\sup_{0\leq\theta\leq
    \delta}\MP^{d}_{x}\big(\norm{X_{\tau_{d}+\theta}^{d}-X_{\tau_{d}}^{d}}>\eta\big)&\leq
                                                                                  \delta
                                                                                  \frac{\hat{K}_{T+\delta}}{\eta^{2}}(1+\norm{x}^{2})\leq
                                                                                  \varepsilon,    
  \end{align*}
  for arbitrary $N_{0}\in\MN$. This proves the second condition above and it
  therefore follows from the Aldous criterion that the sequence 
  $(\MP^{d}_{x})_{d\in\MN}$ is a tight sequence of probability measures on
  $\cB(D(\MRplus,\cHplus))$.       
\end{proof}

\subsection{Weak convergence of the finite-rank operator-valued affine processes}\label{sec:kar22-weak-conv-finite}
In this section we prove weak convergence of the sequence
$(\MP^{d}_{x})_{d\in\MN}$ of laws of $(X^{d})_{d\in\MN}$ given
$X_{0}^{d}=\bP_{d}(x)$ to a unique affine process $X$ with law $\MP_{x}$. By
Proposition~\ref{prop:kar22-tightness} we already know that
$(\MP^{d}_{x})_{d\in\MN}$ is tight, which by the Prokhorov characterization of
relative weak compactness, implies that every subsequence of
$(\MP^{d}_{x})_{d\in\MN}$ admits a weakly convergent subsequence. If we show
that all those convergent subsequences have the same limit $\MP_{x}$, we can
conclude that already $(\MP^{d}_{x})_{d\in\MN}$ converges weakly to $\MP_{x}$,
see also~\cite[Chapter 3]{EK86}. We are thus left with proving
\emph{uniqueness}, which we approach via \emph{martingale problems}. Recall
that for every $d\in\MN$ the process in~\eqref{eq:kar22-Gd-martingale-problem}
is a martingale on $(\Omega,\bar{\cF}^{d},\bar{\MF}^{d},\MP_{x})$, in which
case we say that $X^{d}$, respectively its law $\MP_{x}^{d}$, solves the
\emph{martingale problem} for $\cG^{d}$ with initial condition
$X_{0}^{d}=\bP_{d}(x)$. Next, we formulate a martingale problem for the
operator $\cG$ defined on the set $\set{\E^{-\langle \cdot, u\rangle}\colon
  u\in\cHplus}$ as      
\begin{align}\label{eq:kar22-G-operator}
  \cG\E^{-\langle \cdot,u\rangle}(x)\df \big(F(u)+\langle x,
  R(u)\rangle\big)\E^{-\langle x,u\rangle},\quad x\in\cHplus. 
\end{align}

\begin{definition}\label{def:kar22-martingale-problem} 
  Let $\Omega=D([0,T],\cHplus)$, $\MP$ be a probability measure on
  $\cB(\Omega)$ admitting a canonical process $(X_{t})_{t\geq 0}$. Let $\cG$
  be as in~\eqref{eq:kar22-G-operator} defined on $\cD$ and $x\in\cHplus$. We
  then call $\MP$ a solution to the martingale problem for $\cG$ with initial
  condition $\MP(X_{0}=x)=1$ if for every $f\in\cD$ the process   
  \begin{align}\label{eq:kar22-martingale-problem}
    \left(f(X_{t})-f(x)-\int_{0}^{t}\cG f(X_{s})\D s\right)_{t\geq 0},
  \end{align}
  is a martingale on $(\Omega,\cF,(\cF_{t})_{t\geq 0},\MP)$, where
  $(\cF_{t})_{t\geq 0}$ denotes the natural filtration of $(X_{t})_{t\geq 0}$.  
\end{definition}

That the martingale problem has at least one solution is the assertion of the
following proposition.

\begin{proposition}\label{prop:kar22-weak-convergence}
  Let $x\in\cHplus$. Then every weak limit $\MP_{x}$ of a convergent subsequence
  of $(\MP^{d}_{x})_{d\in\MN}$ solves the martingale problem posed in Definition~\ref{def:kar22-martingale-problem}. Moreover, the canonical
  process of $\MP_{x}$ on $D(\MRplus,\cHplus)$ is continuous in probability.  
\end{proposition}
\begin{proof}
  From the tightness of $(\MP^{d}_{x})_{d\in\MN}$ the existence of a
  weakly convergent subsequence follows from the Prokhorov theorem. Let
  $\MP_{x}$ be such a weak limit of some subsequence
  $(\MP^{d_{n}}_{x})_{n\in\MN}$. We know
  from~\cref{item:kar22-embedding-affine-main-1}, that for every $d\in\MN$
  the process $X^{d}$, respectively its law $\MP^{d}_{x}$, solves the
  martingale problem for $\cG^{d}$ with $X_{0}^{d}=\bP_{d}(x)$, in particular
  this holds for all $(\MP^{d_{n}}_{x})_{n\in\MN}$. Now, note that for every
  $u\in\cHplus$ we have
  \begin{align*}
    \sup_{x\in\cHplus}|\E^{-\langle x,\bP_{d}(u)\rangle}-\E^{-\langle
    x,u\rangle}|&\leq \sup_{x\in\cHplus}\E^{-\langle x,u\rangle}\langle x,
                  u-\bP_{d}(u)\rangle\\
                &\leq \sup_{x\in\cHplus}\E^{-\langle
                  x,u\rangle}\norm{x}\norm{\bP_{d}^{\perp}(u)}\to 0,\quad
                  \text{as }d\to\infty,  
  \end{align*}
  and we also find that
  \begin{align*}
    \sup_{x\in\cHplus}|\cG^{d}\E^{-\langle \cdot, \bP_{d}(u)\rangle}(x)-\cG\E^{-\langle \cdot,u\rangle}(x)|&=\sup_{x\in\cHplus}\big(|\langle x,
                         R_{d}(\bP_{d}(u))-R(u)\rangle|\E^{-\langle
                         x,u\rangle}\big)\\
    &\quad+|F_{d}(\bP_{d}(u))-F(u))|\\
                       &\leq \sup_{x\in\cHplus}\E^{-\langle x,
                         u\rangle}\norm{x}\norm{R_{d}(\bP_{d}(u))-R(u)}\\
                         &\quad+|F_{d}(\bP_{d}(u))-F(u)|\to 0,\quad\text{as }d\to\infty,
  \end{align*}
  where the latter limit holds true as $\sup_{x\in\cHplus}\E^{-\langle
    x,u\rangle}\norm{x}$ is bounded, $F$ and $R$ are continuous on $\cHplus$,
  see Lemma~\ref{lem:kar22-locally-uniform-convergence-Rd}, and
  $\norm{\bP_{d}^{\perp}(u)}=\norm{\bP_{d}(u)-u}\to 0$ as $d\to\infty$. It thus follows from~\cite[Lemma
  5.1]{EK86} that the weak limit $\MP_{x}$ of $(\MP^{d_{n}}_{x})_{n\in\MN}$
  solves the martingale problem in
  Definition~\ref{def:kar22-martingale-problem}. The continuity in probability
  of the canonical process  of $\MP_{x}$ is a consequence of Aldous criterion
  and follows from~\cite[Theorem 3.3.1]{JM86}.   
\end{proof}

Next, we prove that the limit $\MP_{x}$ of every convergent subsequence of
$(\MP_{x}^{d})_{d\in\MN}$ is unique, which in turn
proves~\cref{item:kar22-main-convergence-1} and we also prove the remaining
assertions of Theorem~\ref{thm:kar22-main-convergence}.

\begin{proof}[Proof of Theorem~\ref{thm:kar22-main-convergence}]
  For $x\in\cHplus$, denote by $\MP_{x}$ the limit of some subsequence of
  $(\MP^{d_{n}}_{x})_{n\in\MN}$ and let $X=(X_{t})_{t\geq 0}$ the canonical
  process of $\MP_{x}$ on $\Omega=D(\MRplus,\cHplus)$. By
  Proposition~\ref{prop:kar22-weak-convergence} $(X,\MP_{x})$ is a solution to the
  martingale problem posed in
  Definition~\ref{def:kar22-martingale-problem}. Moreover, let
  $T\geq 0$ arbitrary, $u\in\cHplus$ and define the functions
  $f_u(t,x)\colon [0,T] \times \cHplus \to \MRplus$ by 
  \begin{align*}
    f_u(t,x) = \E^{-\phi(T-t,u)-\langle x, \psi(T-t,u)\rangle},
  \end{align*}
  where $(\phi(\cdot,u),\psi(\cdot,u))$ is the unique solution
  of~\eqref{eq:kar22-Riccati-phi}-\eqref{eq:kar22-Riccati-psi} on $[0,T]$. We
  see that $f_u\in C_b^{1,1}([0,T]\times\cHplus)$ and it thus follows from~\cite[Theorem
  4.7.1]{EK86} that since $X$ solves the martingale problem for
  $\cG$, the process $(X_{t},t)_{t\geq 0}$ solves the associated time-dependent
  martingale problem, i.e. the process
  \begin{align}\label{eq:kar22-time-dependent-2}
    \Big(f_{u}(t,X_{t})-f_{u}(0,x)-\int_{0}^{t}\cG
    f_{u}(s,X_{s})+\frac{\partial}{\partial s}f_{u}(s,X_{s})\,\D s\Big)_{0\leq t\leq T}     
  \end{align}
  is a martingale for every $u\in\cHplus$. Moreover, we see that
  \begin{align}\label{eq:kar22-time-dependent-1}
    \frac{\partial}{\partial t}f_{u}(t,x)&=\Big(\frac{\partial \phi}{\partial
                                           t}(T-t,u)+\langle x,\frac{\partial \psi}{\partial t}(T-t,u)\rangle \Big)f_{u}(t,x)\nonumber\\
                                         &=\left(F(\psi(T-t,u))+\langle x, R(\psi(T-t,u))\rangle\right) f_{u}(t,x)\,,
  \end{align} 
  which inserted into~\eqref{eq:kar22-time-dependent-2} nullifies the term
  $\cG f_{u}(s,X_{s})$, compare with~\eqref{eq:kar22-G-operator}. Hence we see
  that the process $(f_{u}(t,X_{t})-f_{u}(0,x))_{t\leq T}$ must be a
  martingale. This implies in particular, that
  $\EXspec{\MP_{x}}{f_{u}(t,X_{t})}=f_{u}(0,x)$ for all $0\leq t\leq T$,
  i.e. for $t=T$ we obtain
  \begin{align*}
   \EXspec{\MP_{x}}{f_{u}(T,X_{T})}=\EXspec{\MP_{x}}{\E^{\langle X_{T},
    u\rangle}}=\E^{-\phi(T,u)-\langle x, \psi(T,u)\rangle},\quad u\in\cHplus.  
  \end{align*}
  Since $T>0$ was aribtrary, this implies that the affine transform
  formula~\eqref{eq:kar22-affine-Galerkin-theorem} holds true.  Note, that
  since this holds for every $u\in\cHplus$ and the Laplace transform is measure
  determining on $\cHplus$, see~\cite[Lemma A.1]{CKK22b}, it follows that
  $X_{t}$ is unique in law for every fixed $t\geq 0$. But since the process
  $X$ is the solution to the martingale problem in
  Definition~\ref{def:kar22-martingale-problem}, it follows from~\cite[Theorem
  4.4.2 (a)]{EK86} that the pointwise uniqueness already implies the uniqueness in
  distribution (i.e. uniqueness of the solution to the martingale problem on $[0,T]$). Again since $T$
  was arbitrary, this then proves~\cref{item:kar22-main-convergence-1}.\\
  Next we show~\cref{item:kar22-main-convergence-2}. Note that in the first
  part we just proved that the limit of convergent subsequences of
  $(\MP_{x}^{d})_{d\in\MN}$ is given by $\MP_{x}$ and that the associated
  process $X$ satisfies the affine transform formula, i.e. the sequence
  $(X^{d})_{d\in\MN}$ converges weakly to $X$ on $D(\MRplus,\cHplus)$.
  We can thus continue with the convergence rate
  in~\eqref{eq:kar22-main-convergence-rate}. First, note that by standard
  estimates it follows
  from~\eqref{eq:kar22-affine-Galerkin-theorem},~\eqref{eq:kar22-affine-Galerkin}
  that
    \begin{align*}
      \left|\EXspec{\MP_{x}}{\E^{-\langle u,
      X_{t}\rangle}}-\EXspec{\MP^{d}_{x}}{\E^{-\langle u,
      X^{d}_{t}\rangle}}\right|&\leq
      |\phi(t,u)-\phi_{d}(t,\bP_{d}(u))|\\
      &\quad+\norm{x}\norm{\psi(t,u)-\psi_{d}(t,\bP_{d}(u))},  
    \end{align*}
    and it thus follows from Corollary~\ref{cor:kar22-convergence-Galerkin-compact}
    that there exists a $C_{T}$ independent of $d\in\MN$, such that
  \begin{align}
    \sup_{t\in [0,T],\norm{u}_{\cV}\leq 1}
    \left|\EXspec{\MP_{x}}{\E^{-\langle u,X_{t}\rangle}}-\EXspec{\MP^{d}_{x}}{\E^{-\langle u,X^{d}_{t}\rangle}}\right|&\leq C_{T}\norm{\bP^{\perp}_{d}}_{\cL(\cV,\cH)}(1+\norm{x}),
  \end{align}
  Note tat if the conditions $\norm{\mu(\cHpluso)}_{\cV}<\infty$ and
  $B^{*}(\cV_{0})\subseteq \cV_{0}$ do not hold then we see that the
  convergence rate in part i) of Remark~\ref{rem:kar22-main-convergence}
  follows from Proposition~\ref{prop:kar22-existence-Riccati} instead of
  Corollary~\ref{cor:kar22-convergence-Galerkin-compact}.
\end{proof}

\begin{remark}
Note that if $\norm{\mu(\cHpluso)}_{\cV}<\infty$ and $B^{*}(\cV_{0})\subseteq
\cV_{0}$ do not hold, then there still exists a constant $\tilde{K}$,
independent of $d\in\MN$, such that~\eqref{eq:kar22-main-convergence-rate}
holds with right-hand side $\tilde{K}C_{T,d}(1+\norm{x})$ with $C_{T,d}$
in~\eqref{eq:kar22-convergence-Galerkin-CTd}.   
\end{remark}

\bibliographystyle{unsrt}
\bibliography{literatur}
\end{document}